\documentclass{amsart}
\usepackage{amsmath,amsthm,fixltx2e}
\usepackage{color}
\usepackage{amsxtra}
\usepackage{amscd}
\usepackage{amsfonts}
\usepackage{amssymb}
\usepackage{eucal}

\newtheorem{theorem}{Theorem}[section]
\newtheorem{cor}[theorem]{Corollary}
\newtheorem{lem}[theorem]{Lemma}
\newtheorem{prop}[theorem]{Proposition}

\newtheorem{thm}[theorem]{Theorem}

\newtheorem{rem}[theorem]{Remark}

\newtheorem{defn}[theorem]{Definition}

\newcommand{\nc}{\newcommand}

\nc\ol{\overline} \nc\ul{\underline} \nc\wt{\widetilde}
\nc{\z}{\zeta}

\nc{\ZZ}{{\mathbb Z}} \nc{\NN}{{\mathbb N}} \nc{\CC}{{\mathbb C}}
\nc{\QQ}{{\mathbb Q}} \nc{\CP}{{\mathbb {CP}}} \nc{\A}{{\mathcal A}}
\nc\U{{\mathfrak U}} \nc{\F}{{\mathcal F}} \nc{\N}{{\mathcal N}}
\nc{\Aa}{{\mathcal A}} \nc{\E}{{\mathcal E}} \nc{\sS}{{\mathbb S}}
\nc{\K}{{\mathcal K}} \nc{\Ll}{{\mathcal L}} \nc{\Y}{{\mathcal Y}}
\nc{\SSS}{{\mathcal S}}

\newcommand\C{\mathfrak c}

\newcommand{\gl}{\mathfrak{gl}}
\newcommand{\ssl}{\mathfrak{sl}}

\newcommand{\Sym}{\mathrm{Sym}}

\newcommand{\drj}{\mathrm{DJ}}

\nc{\iso}{{\stackrel{\sim}{\longrightarrow}}}

\begin{document}

\author[Alexander Tsymbaliuk]{Alexander Tsymbaliuk}
\address{A.~Tsymbaliuk: Yale University, Department of Mathematics, New Haven, CT 06511, USA}
\email{oleksandr.tsymbaliuk@yale.edu}

\title[Several realizations of Fock modules for toroidal $\ddot{U}_{q,d}(\ssl_n)$]
{Several realizations of Fock modules for toroidal $\ddot{U}_{q,d}(\ssl_n)$}

\begin{abstract}
In this paper, we relate the well-known
\emph{Fock representations} of $\ddot{U}_{q,d}(\ssl_n)$ to the
vertex, shuffle, and `$L$-operator' representations of $\ddot{U}_{q,d}(\ssl_n)$.
These identifications generalize those for the quantum toroidal algebra of $\gl_1$,
which were recently established in~\cite{FJMM2}.
\end{abstract}

\maketitle


\section*{Introduction}

 In the recent paper~\cite{FJMM2}, authors proposed a shuffle
approach to the Bethe ansatz problem for certain modules over
the quantum toroidal algebra of $\gl_1$, viewing the latter as the
Drinfeld double of the \emph{small shuffle} algebra.
The general idea behind a shuffle approach is that it frequently allows to interpret
complicated concepts in simple terms.
As the representation theory of quantum toroidal algebras of $\ssl_n$ is quite
similar to that of quantum toroidal algebras of $\gl_1$
(though technically it is more involved), it is desirable
to generalize the aforementioned construction for the former case.

 In this article, we identify different families of representations of
quantum toroidal algebras of $\ssl_n$.
This will be crucial for our arguments in~\cite{FT2}, where we diagonalize the commutative
subalgebras of the quantum toroidal algebras of $\ssl_n$ studied in~\cite{FT1}.

 This paper is organized as follows:

$\bullet$
 In Section 1, we recall the definition and key results about the
quantum toroidal algebra $\ddot{U}_{q,d}(\ssl_n), n\geq 3$.
In particular, we recall the relation to the shuffle algebra $S$ (of $A^{(1)}_{n-1}$-type) studied in~\cite{N, FT1}.

 We also discuss three different constructions of their representations:

-combinatorial representations $\tau^p_{u,\bar{c}}$ introduced in~\cite{FJMM1},

-vertex representations $\rho^p_{u,\bar{c}}$ constructed in~\cite{S},

-shuffle representations $\pi^p_{u,\bar{c}}$ introduced in this paper.

\noindent
 Our construction of $\pi^p_{u,\bar{c}}$ is similar to that of~\cite{FJMM2} for the quantum toroidal algebra of $\gl_1$.
In particular, the underlying vector space $S_{1,p}(u)$ carries a natural $S$-bimodule structure, while $\pi^p_{u,\bar{c}}$
is the extension of the left $S$-action to the action of $\ddot{U}_{q,d}(\ssl_n)$, see Proposition~\ref{extension to toroidal action}.

$\bullet$
 In Section 2, we relate the aforementioned three different families of representations:

-In Theorem~\ref{main1}, we show that $\pi^p_{u,\bar{c}}$ induces an action of $\ddot{U}_{q,d}(\ssl_n)$ on the factor of $S_{1,p}(u)$ by the right
$S'$-action (here $S'\subset S$ denotes the augmentation ideal), which is isomorphic to the $\tau^p_{u,\bar{c}}$-action.
In Theorems~\ref{main2},~\ref{main3}, we generalize this result to some other families of representations.

-In Theorem~\ref{main4}, we show that Miki's isomorphism $\varpi$ of the quantum
toroidal algebras intertwines the dual of the combinatorial representation $\tau^p_{u,\bar{c}}$
and the corresponding vertex representation $\rho^p_{u',\bar{c}'}$ for appropriate parameters.

$\bullet$
 In Section 3, we study the matrix elements of $L$ operators associated to the vertex representations $\rho^p_{u,\bar{c}}$.
In Theorem~\ref{main5}, we derive an explicit formula for the matrix element $L^{p,\bar{c}}_{\emptyset,\emptyset}$,
whose shuffle realization was obtained in~\cite{FT1}.
This allows us to identify the shuffle $S$-bimodule $S_{1,p}(u)$ with
the $S$-bimodule generated by $L^{p,\bar{c}}_{\emptyset,\emptyset}$, see Proposition~\ref{L-opertaor realization}.


\subsection*{Acknowledgments}
$\ $

 I would like to thank P.~Etingof, B.~Feigin, M.~Finkelberg, and A.~Negut for many stimulating discussions over the years.
I am indebted to the anonymous referee for insightful comments on the first version of the paper, which led to a better exposition of the material.

I would like to thank the Max Planck Institute for Mathematics in Bonn
for the hospitality and support in June 2015, where part of this project was carried out.
The author also gratefully acknowledges support from the Simons Center for Geometry and Physics, Stony Brook University,
at which most of the research for this paper was performed, as well as Yale University, where the final version
of this paper was completed.

This work was partially supported by the NSF Grants DMS--1502497, DMS--1821185.


\section{Basic definitions and constructions}


\subsection{Quantum toroidal algebras of $\ssl_n$ for $n\geq 3$}
$\ $

 Let $q,d\in \CC^\times$ be two parameters. We set $[n]:=\{0,1,\ldots,n-1\},\ [n]^\times:=[n]\backslash \{0\}$,
the former viewed as a set of mod $n$ residues.
Let $g_m(z):=\frac{q^mz-1}{z-q^m}$. Define $\{a_{i,j},m_{i,j}|i,j\in [n]\}$ by
  $$a_{i,i}=2,\ a_{i,i\pm 1}=-1,\ m_{i,i\pm1}=\mp 1,\ \mathrm{and}\ a_{i,j}=m_{i,j}=0\ \mathrm{otherwise}.$$
The quantum toroidal algebra of $\ssl_n$, denoted by $\ddot{U}_{q,d}(\ssl_n)$,
is the unital associative $\CC$-algebra generated by
$\{e_{i,k}, f_{i,k}, \psi_{i,k}, \psi_{i,0}^{-1}, \gamma^{\pm 1/2}, q^{\pm d_1}, q^{\pm d_2}\}_{i\in [n]}^{k\in \ZZ}$
with the following defining relations:
\begin{equation}\tag{T0.1} \label{T0.1}
  [\psi_i^\pm(z),\psi_j^\pm(w)]=0,\ \gamma^{\pm 1/2}-\mathrm{central},
\end{equation}
\begin{equation}\tag{T0.2} \label{T0.2}
  \psi_{i,0}^{\pm 1}\cdot \psi_{i,0}^{\mp 1}=\gamma^{\pm 1/2}\cdot \gamma^{\mp 1/2}=q^{\pm d_1}\cdot q^{\mp d_1}= q^{\pm d_2}\cdot q^{\mp d_2}=1,
\end{equation}
\begin{equation}\tag{T0.3}\label{T0.3}
  q^{d_1}e_i(z)q^{-d_1}=e_i(qz),\ q^{d_1}f_i(z)q^{-d_1}=f_i(qz),\ q^{d_1}\psi^\pm_i(z)q^{-d_1}=\psi^\pm_i(qz),
\end{equation}
\begin{equation}\tag{T0.4}\label{T0.4}
  q^{d_2}e_i(z)q^{-d_2}=qe_i(z),\ q^{d_2}f_i(z)q^{-d_2}=q^{-1}f_i(z),\ q^{d_2}\psi^\pm_i(z)q^{-d_2}=\psi^\pm_i(z),
\end{equation}
\begin{equation}\tag{T1} \label{T1}
  g_{a_{i,j}}(\gamma^{-1}d^{m_{i,j}}z/w)\psi_i^+(z)\psi_j^-(w)=g_{a_{i,j}}(\gamma d^{m_{i,j}}z/w)\psi_j^-(w)\psi_i^+(z),
\end{equation}
\begin{equation}\tag{T2} \label{T2}
  e_i(z)e_j(w)=g_{a_{i,j}}(d^{m_{i,j}}z/w)e_j(w)e_i(z),
\end{equation}
\begin{equation}\tag{T3}\label{T3}
  f_i(z)f_j(w)=g_{a_{i,j}}(d^{m_{i,j}}z/w)^{-1}f_j(w)f_i(z),
\end{equation}
\begin{equation}\tag{T4}\label{T4}
  (q-q^{-1})[e_i(z),f_j(w)]=\delta_{i,j}\left(\delta(\gamma w/z)\psi_i^+(\gamma^{1/2}w)-\delta(\gamma z/w)\psi_i^-(\gamma^{1/2}z)\right),
\end{equation}
\begin{equation}\tag{T5}\label{T5}
  \psi_i^\pm(z)e_j(w)=g_{a_{i,j}}(\gamma^{\pm 1/2}d^{m_{i,j}}z/w)e_j(w)\psi_i^\pm(z),
\end{equation}
\begin{equation}\tag{T6}\label{T6}
  \psi_i^{\pm}(z)f_j(w)=g_{a_{i,j}}(\gamma^{\mp 1/2}d^{m_{i,j}}z/w)^{-1}f_j(w)\psi_i^\pm(z),
\end{equation}
\begin{equation}\tag{T7.1}\label{T7.1}
  \mathrm{Sym}_{z_1,z_2}\ [e_i(z_1),[e_i(z_2),e_{i\pm 1}(w)]_q]_{q^{-1}}=0,\
  [e_i(z), e_j(w)]=0\ \mathrm{for}\ j\ne i, i\pm 1,
\end{equation}
\begin{equation}\tag{T7.2}\label{T7.2}
  \mathrm{Sym}_{z_1,z_2}\ [f_i(z_1),[f_i(z_2),f_{i\pm 1}(w)]_q]_{q^{-1}}=0,\
  [f_i(z), f_j(w)]=0\ \mathrm{for}\ j\ne i, i\pm 1,
\end{equation}
where we set $[a,b]_x:=ab-x\cdot ba$ and define the generating series as follows:
  $$e_i(z):=\sum_{k=-\infty}^{\infty}{e_{i,k}z^{-k}},\  \
    f_i(z):=\sum_{k=-\infty}^{\infty}{f_{i,k}z^{-k}},\  \
    \psi_i^{\pm}(z):=\psi_{i,0}^{\pm 1}+\sum_{ r>0}{\psi_{i,\pm r}z^{\mp r}},\ \
    \delta(z):=\sum_{k=-\infty}^{\infty}{z^k}.$$

 It will be convenient to use the generators $\{h_{i,k}\}_{k\ne 0}$ instead of  $\{\psi_{i,k}\}_{k\ne 0}$, defined by
  $$\exp\left(\pm(q-q^{-1})\sum_{r>0}h_{i,\pm r}z^{\mp r}\right)=\bar{\psi}_i^\pm(z):=\psi_{i,0}^{\mp 1}\psi^\pm_i(z),\ \
     h_{i,\pm r}\in \CC[\psi_{i,0}^{\mp 1},\psi_{i,\pm 1},\psi_{i,\pm2}, \ldots].$$
Then the relations (T5,T6) are equivalent to the following (we use  $[m]_q:=(q^m-q^{-m})/(q-q^{-1})$):
\begin{equation}\tag{T5$'$}\label{T5'}
  \psi_{i,0}e_{j,l}=q^{a_{i,j}}e_{j,l}\psi_{i,0},\
  [h_{i, k}, e_{j,l}]=d^{-km_{i,j}}\gamma^{-|k|/2}\frac{[ka_{i,j}]_q}{k}e_{j,l+k}\ \mathrm{for}\ k\ne 0,
\end{equation}
\begin{equation}\tag{T6$'$}\label{T6'}
  \psi_{i,0}f_{j,l}=q^{-a_{i,j}}f_{j,l}\psi_{i,0},\
  [h_{i, k},  f_{j,l}]=-d^{-km_{i,j}}\gamma^{|k|/2}\frac{[ka_{i,j}]_q}{k}f_{j,l+k}\ \mathrm{for}\ k\ne 0.
\end{equation}
We also introduce $h_{i,0},c,c'$ via
$\psi_{i,0}=q^{h_{i,0}},\ \gamma^{1/2}=q^c,\ c'=\sum_{i\in [n]} h_{i,0}$,
so that $c,c'$ are central.

 In Sections 1.3--1.4, we will also need to make sense of the elements
$q^{\frac{h_{i,0}}{2n}}, \gamma^{\frac{1}{2n}}, q^{\frac{d_2}{n}}$.
In such cases, we formally add elements of the form
$q^{\frac{h_{i,0}}{N}},q^{\frac{c}{N}},q^{\frac{d_1}{N}},q^{\frac{d_2}{N}}$ for any $N\in \ZZ_{>0}$.


\subsection{Hopf algebra structure, Hopf pairing, and Drinfeld double}
$\ $

 Following~\cite{FT1}, we recall some of the basic results on $\ddot{U}_{q,d}(\ssl_n)$ which are relevant to us.

\noindent
$\bullet$ \emph{Topological Hopf algebra structure on $\ddot{U}_{q,d}(\ssl_n)$.}

 Following~\cite[Theorem 2.1]{DI}, we endow $\ddot{U}_{q,d}(\ssl_n)$ with a topological Hopf algebra
structure by defining the coproduct $\Delta$, the counit $\epsilon$, and the antipode $S$ as follows:
\begin{multline}\tag{H1}\label{coproduct}
  \Delta(\psi^\pm_i(z))=\psi^\pm_i(\gamma^{\pm 1/2}_{(2)}z)\otimes \psi^\pm_i(\gamma^{\mp 1/2}_{(1)}z),\
  \Delta(x)=x\otimes x\ \mathrm{for}\ x=\gamma^{\pm 1/2}, q^{\pm d_1}, q^{\pm d_2},\\
  \Delta(e_i(z))=e_i(z)\otimes 1+ \psi^-_i(\gamma^{1/2}_{(1)}z)\otimes e_i(\gamma_{(1)}z),\
  \Delta(f_i(z))=1\otimes f_i(z)+f_i(\gamma_{(2)}z)\otimes  \psi_i^+(\gamma^{1/2}_{(2)}z),
\end{multline}
\begin{equation}\tag{H2}\label{counit}
  \epsilon(e_i(z))=\epsilon(f_i(z))=0,\ \epsilon(\psi^{\pm}_i(z))=1,\
  \epsilon(x)=1\ \mathrm{for}\ x=\gamma^{\pm 1/2}, q^{\pm d_1}, q^{\pm d_2},
\end{equation}
\begin{multline}\tag{H3}\label{antipode}
  S(e_i(z))=-\psi^-_i(\gamma^{-1/2}z)^{-1}e_i(\gamma^{-1}z),\
  S(f_i(z))=-f_i(\gamma^{-1}z)\psi^+_i(\gamma^{-1/2}z)^{-1},\\
  S(x)=x^{-1}\ \mathrm{for}\ x=\gamma^{\pm 1/2}, q^{\pm d_1}, q^{\pm d_2},\psi^\pm_i(z),
\end{multline}
where $\gamma_{(1)}:=\gamma\otimes 1$ and $\gamma_{(2)}:=1\otimes \gamma$.

\noindent
$\bullet$ \emph{Sub/quotient-algebras of $\ddot{U}_{q,d}(\ssl_n)$.}

 In what follows, we will need the following subalgebras of $\ddot{U}_{q,d}(\ssl_n)$:

\noindent
  $\circ$ $\ddot{U}^{\geq}$ is the subalgebra of $\ddot{U}_{q,d}(\ssl_n)$ generated by
  $\{e_{i,k},\psi_{i,l}, \psi_{i,0}^{-1}, \gamma^{\pm 1/2}, q^{\pm d_1}, q^{\pm d_2}\}_{i\in [n]}^{k\in \ZZ,l\in -\NN}$.

\noindent
  $\circ$ $\ddot{U}^{\leq}$ is the subalgebra of $\ddot{U}_{q,d}(\ssl_n)$ generated by
  $\{f_{i,k},\psi_{i,l}, \psi_{i,0}^{-1}, \gamma^{\pm 1/2}, q^{\pm d_1}, q^{\pm d_2}\}_{i\in [n]}^{k\in \ZZ,l\in \NN}$.

\noindent
  $\circ$ $\ddot{U}^+,\ddot{U}^-$ are the subalgebras of $\ddot{U}_{q,d}(\ssl_n)$ generated by
  $\{e_{i,k}\}_{i\in [n]}^{k\in \ZZ}$ and $\{f_{i,k}\}_{i\in [n]}^{k\in \ZZ}$, respectively.

\noindent
 $\circ$ $\ddot{U}^0$ is the subalgebra of $\ddot{U}_{q,d}(\ssl_n)$ generated by
  $\{\psi_{i,k},\psi^{-1}_{i,0},\gamma^{\pm 1/2},q^{\pm d_1}, q^{\pm d_2}\}_{i\in [n]}^{k\in \ZZ}$.

 We also define two modifications of $\ddot{U}_{q,d}(\ssl_n)$:

\noindent
$\circ$ Let $\ddot{U}^{'}_{q,d}(\ssl_n)$ be obtained from $\ddot{U}_{q,d}(\ssl_n)$
by ``ignoring'' the generator $q^{\pm d_2}$ and taking a quotient by the ideal $(c')$, i.e., setting $c'=0$.
The subalgebras $\ddot{U}^{'\geq}, \ddot{U}^{'\leq}, \ddot{U}^{'\pm}, \ddot{U}^{'0}$ of $\ddot{U}^{'}_{q,d}(\ssl_n)$
are defined completely analogously to $\ddot{U}^{\geq}, \ddot{U}^{\leq},\ddot{U}^{\pm},\ddot{U}^{0}$ above.

\noindent
$\circ$ Let $^{'}\ddot{U}_{q,d}(\ssl_n)$ be obtained from $\ddot{U}_{q,d}(\ssl_n)$
by ``ignoring'' the generator $q^{\pm d_1}$ and taking a quotient by the ideal $(c)$, i.e., setting $c=0$.
The subalgebras ${^{'}\ddot{U}^{\geq}}, {^{'}\ddot{U}^{\leq}},{^{'}\ddot{U}^{\pm}},{^{'}\ddot{U}^{0}}$ of $^{'}\ddot{U}_{q,d}(\ssl_n)$
are defined completely analogously to $\ddot{U}^{\geq}, \ddot{U}^{\leq},\ddot{U}^{\pm},\ddot{U}^{0}$  above.

\medskip
\noindent
$\bullet$ \emph{Hopf pairing and a Drinfeld double realization of $\ddot{U}_{q,d}(\ssl_n)$.}

  Analogously to the case of quantum affine algebras (see~\cite{G}), we have the following result.

\begin{thm}\label{Drinfeld double sln}
(a)  There exists a unique Hopf algebra  pairing $\varphi\colon \ddot{U}^\geq \times \ddot{U}^\leq\to \CC$  satisfying
\begin{equation*}
  \varphi(e_i(z), f_j(w))=\frac{\delta_{i,j}}{q-q^{-1}}\cdot \delta\left(\frac{z}{w}\right),
  \varphi(\psi^-_i(z),\psi^+_j(w))=g_{a_{i,j}}(d^{m_{i,j}}z/w),
  \varphi(q^{d_2},q^{d_2})=q^{\frac{n(n^2-1)}{12}},
\end{equation*}
\begin{equation*}
  \varphi(e_i(z),x^-)=\varphi(x^+,f_i(z))=0\ \mathrm{for}\  x^\pm=\psi^\mp_j(w), \gamma^{1/2}, q^{d_1}, q^{d_2},
\end{equation*}
\begin{equation*}
  \varphi(\gamma^{1/2}, q^{d_1})=\varphi(q^{d_1},\gamma^{1/2})=q^{-1/2},\
  \varphi(\psi^-_i(z),q^{d_2})=q^{-1},\ \varphi(q^{d_2},\psi^+_i(z))=q,
\end{equation*}
\begin{equation*}
  \varphi(\psi^-_i(z),x)=\varphi(x,\psi^+_i(z))=1\ \mathrm{for}\ x=\gamma^{1/2}, q^{d_1},
\end{equation*}
\begin{equation*}
  \varphi(\gamma^{1/2}, q^{d_2})=\varphi(q^{d_2},\gamma^{1/2})=\varphi(\gamma^{1/2},\gamma^{1/2})=\varphi(q^{d_1},q^{d_1})=
  \varphi(q^{d_1},q^{d_2})=\varphi(q^{d_2},q^{d_1})=1.
\end{equation*}
\noindent
(b) The natural Hopf algebra homomorphism from the Drinfeld double $D_\varphi(\ddot{U}^\geq, \ddot{U}^\leq)$
to $\ddot{U}_{q,d}(\ssl_n)$ induces the isomorphism
   $$\Xi\colon D_\varphi(\ddot{U}^\geq, \ddot{U}^\leq)/I\iso \ddot{U}_{q,d}(\ssl_n)\ \mathrm{with}\
     I:=(x\otimes 1-1\otimes x|x=\psi_{i,0}^{\pm 1}, \gamma^{\pm 1/2}, q^{\pm d_1}, q^{\pm d_2})_{i\in [n]}.$$
\noindent
(c) Analogously to (b), the algebras $\ddot{U}^{'}_{q,d}(\ssl_n)$ and  $^{'}\ddot{U}_{q,d}(\ssl_n)$
admit the Drinfeld double realizations via
$D_{\varphi'}(\ddot{U}^{'\geq},\ddot{U}^{'\leq})$ and $D_{'\varphi}({^{'}\ddot{U}^{\geq}},{^{'}\ddot{U}^{\leq}})$,
where $\varphi'$ and $'\varphi$ are defined similarly to $\varphi$.

\noindent
(d)  The pairings $\varphi, \varphi',$$'\varphi$ are nondegenerate if and only if $q,qd,qd^{-1}$ are not roots of unity.

\noindent
(e) If $q,qd,qd^{-1}$ are not roots of unity, then the algebras
$\ddot{U}_{q,d}(\ssl_n), \ddot{U}^{'}_{q,d}(\ssl_n)$ and $^{'}\ddot{U}_{q,d}(\ssl_n)$ admit the universal $R$-matrices
$R, R^{'}$ and ${^{'}R}$, associated to the pairings $\varphi, \varphi'$ and $'\varphi$, respectively.
\end{thm}


\subsection{Two copies of $U_q(\widehat{\ssl}_n)$ inside $\ddot{U}_{q,d}(\ssl_n)$}\label{two copies}
$\ $

 Let $U_q(\widehat{\ssl}_n)$ be the quantum affine algebra of $\ssl_n$ presented in the \emph{new Drinfeld realization}, see \cite{Dr2}.
This is the unital associative $\CC$-algebra generated by
$\{e_{i,k}, f_{i,k}, \psi_{i,k}, \psi_{i,0}^{-1}, C^{\pm 1}, \wt{D}^{\pm 1}\}_{i\in [n]^\times}^{k\in \ZZ}$
with the defining relations similar to those of $\ddot{U}_{q,d}(\ssl_n)$ (our notation follow~\cite{M2}):
\begin{equation}\tag{A0.1} \label{A0.1}
  [\psi_i^\pm(z),\psi_j^\pm(w)]=0,\ C^{\pm 1}-\mathrm{central},
\end{equation}
\begin{equation}\tag{A0.2} \label{A0.2}
  \psi_{i,0}^{\pm 1}\cdot \psi_{i,0}^{\mp 1}=C^{\pm 1}\cdot C^{\mp 1}=\wt{D}^{\pm 1}\cdot \wt{D}^{\mp 1}=1,
\end{equation}
\begin{equation}\tag{A0.3}\label{A0.3}
  \wt{D}e_i(z)\wt{D}^{-1}=qe_i(q^{-n}z),\ \wt{D}f_i(z)\wt{D}^{-1}=q^{-1}f_i(q^{-n}z),\ \wt{D}\psi^\pm_i(z)\wt{D}^{-1}=\psi^\pm_i(q^{-n}z),
\end{equation}
\begin{equation}\tag{A1} \label{A1}
  g_{a_{i,j}}(C^{-1}z/w)\psi^+_i(z)\psi^-_j(w)=\psi^-_j(w)\psi^+_i(z)g_{a_{i,j}}(Cz/w),
\end{equation}
\begin{equation}\tag{A2} \label{A2}
  e_i(z)e_j(w)=g_{a_{i,j}}(z/w)e_j(w)e_i(z),
\end{equation}
\begin{equation}\tag{A3}\label{A3}
  f_i(z)f_j(w)=g_{a_{i,j}}(z/w)^{-1}f_j(w)f_i(z),
\end{equation}
\begin{equation}\tag{A4}\label{A4}
  (q-q^{-1})[e_i(z),f_j(w)]=\delta_{i,j}\left(\delta(C w/z)\psi_i^+(Cw)-\delta(C z/w)\psi_i^-(Cz)\right),
\end{equation}
\begin{equation}\tag{A5}\label{A5}
   \psi_i^+(z)e_j(w)=g_{a_{i,j}}(z/w)e_j(w)\psi_i^+(z),\
   \psi_i^-(z)e_j(w)=g_{a_{i,j}}(C^{-1}z/w)e_j(w)\psi_i^-(z),\
\end{equation}
\begin{equation}\tag{A6}\label{A6}
   \psi_i^+(z)f_j(w)=g_{a_{i,j}}(C^{-1}z/w)^{-1}f_j(w)\psi_i^+(z),\
   \psi_i^-(z)f_j(w)=g_{a_{i,j}}(z/w)^{-1}f_j(w)\psi_i^-(z),\
\end{equation}
\begin{equation}\tag{A7.1}\label{A7.1}
  \mathrm{Sym}_{z_1,z_2}\ [e_i(z_1),[e_i(z_2),e_j(w)]_q]_{q^{-1}}=0\ \mathrm{if}\ a_{i,j}=-1,\
  [e_i(z), e_j(w)]=0\ \mathrm{if}\ a_{i,j}=0,
\end{equation}
\begin{equation}\tag{A7.2}\label{A7.2}
  \mathrm{Sym}_{z_1,z_2}\ [f_i(z_1),[f_i(z_2),f_j(w)]_q]_{q^{-1}}=0\ \mathrm{if}\ a_{i,j}=-1,\
  [f_i(z), f_j(w)]=0\ \mathrm{if}\ a_{i,j}=0,
\end{equation}
where the generating series $e_i(z), f_i(z), \psi^\pm_i(z)$ are defined as before.

 This algebra is known to admit a classical \emph{Drinfeld--Jimbo realization} of~\cite{Dr1,Jim}.
To state this explicitly, let $U_q^{\drj}(\widehat{\ssl}_n)$ be the unital associative $\CC$-algebra
generated by $\{x^\pm_i,t_i^{\pm 1},D^{\pm 1}\}_{i\in [n]}$
with the following defining relations:
  $$D^{\pm 1} D^{\mp 1}=1,\ Dt_iD^{-1}=t_i,\ Dx^\pm_iD^{-1}=q^{\pm 1}x^\pm_i,$$
  $$t_i^{\pm 1}t_i^{\mp 1}=1,\ t_it_j=t_jt_i,\ t_ix^\pm_jt_i^{-1}=q^{\pm a_{i,j}}x^\pm_j,$$
  $$[x^+_i,x^-_j]=\delta_{i,j}\cdot \frac{t_i-t_i^{-1}}{q-q^{-1}},\
    \sum_{s=0}^{1-a_{i,j}}\frac{(-1)^s}{[s]_q! [1-a_{i,j}-s]_q!} (x^\pm_i)^sx^\pm_j(x^\pm_i)^{1-a_{i,j}-s}=0\ (i\ne j),$$
where $[m]_q!:=[m]_q [m-1]_q\cdots[1]_q$.

 According to~\cite{Dr2}, there is a $\CC(q)$-algebra isomorphism
$\Phi\colon U_q^\drj(\widehat{\ssl}_n)\iso U_q(\widehat{\ssl}_n)$ given by
 $$x^+_i\mapsto e_{i,0},\ x^-_i\mapsto f_{i,0},\ t_i^{\pm 1}\mapsto \psi_{i,0}^{\pm 1}\ (1\leq i\leq n-1),\
   t_0\mapsto C\cdot (\psi_{1,0}\cdots \psi_{n-1,0})^{-1},\ D\mapsto \widetilde{D},$$
 $$x^+_0\mapsto C (\psi_{1,0}\cdots \psi_{n-1,0})^{-1}[\cdots [f_{1,1},f_{2,0}]_q,\cdots,f_{n-1,0}]_q,$$
 $$x^-_0\mapsto [e_{n-1,0},\cdots ,[e_{2,0},e_{1,-1}]_{q^{-1}}\cdots]_{q^{-1}} (\psi_{1,0}\cdots \psi_{n-1,0})C^{-1}.$$

\begin{rem}
(a) The above isomorphism was stated without a proof in~\cite{Dr2}. The inverse $\Phi^{-1}$ was constructed
in~\cite{B} by using the braid group action on $U_q^\drj(\widehat{\ssl}_n)$ due to G.~Lusztig.
The direct verification of the fact that the above assignment gives rise to a $\CC(q)$-algebra homomorphism
$\Phi\colon U_q^\drj(\widehat{\ssl}_n)\iso U_q(\widehat{\ssl}_n)$ was given in~\cite{Jin} by utilizing the technique of
$q$-commutators, which also plays a key computational role in the current paper. However, proofs of injectivity
of $\Phi^{-1}$ and $\Phi$ in~\cite{B,Jin} had a gap, which was only recently filled in~\cite{Da}.

\noindent
(b) In the classical literature, the \emph{grading} elements $D,\widetilde{D}$ satisfy slightly different relations, while
our conventions are better adapted to fit into the toroidal story and follow that of~\cite{M2}.
\end{rem}

 Following~\cite{VV}, we introduce the \emph{vertical} and \emph{horizontal}
copies of $U_q(\widehat{\ssl}_n)$ inside $\ddot{U}_{q,d}(\ssl_n)$.
Consider two algebra homomorphisms $h,v\colon U_q(\widehat{\ssl}_n)\to \ddot{U}_{q,d}(\ssl_n)$ defined by
  $$h\colon x^+_i\mapsto e_{i,0},\ x^-_i\mapsto f_{i,0},\ t_i\mapsto \psi_{i,0},\ D\mapsto q^{d_2},$$
  $$v\colon e_{i,k}\mapsto d^{ik}e_{i,k},\ f_{i,k}\mapsto d^{ik}f_{i,k},\ \psi_{i,k}\mapsto d^{ik}\gamma^{k/2}\psi_{i,k},\
    C\mapsto \gamma,\ \widetilde{D}\mapsto q^{-nd_1}\cdot q^{\sum_{j=1}^{n-1}\frac{j(n-j)}{2}h_{j,0}},$$
where we follow the conventions of Section 1.1 and add elements $q^{h_{j,0}/2}$ to $\ddot{U}_{q,d}(\ssl_n)$.

 According to~\cite{VV}, both $h,v$ are inclusions.
The images of $h$ and $v$, denoted by $\dot{U}^{\mathrm{h}}_q(\ssl_n)$ and $\dot{U}^{\mathrm{v}}_q(\ssl_n)$,
are called the \emph{horizontal} and \emph{vertical} copies of $U_q(\widehat{\ssl}_n)$ inside $\ddot{U}_{q,d}(\ssl_n)$.

\begin{rem}
The injectivity of $h,v$ was stated in~\cite{VV} without a proof, and was used in numeric literature afterwards.
A simple way to see the injectivity is to use the double realization of all algebras involved
(here $U_q(\widehat{\ssl}_n)$ is treated as in Theorem~\ref{Drinfeld double sln}, while $U_q^\drj(\widehat{\ssl}_n)$
is treated with respect to the Drinfeld-Jimbo Borel subalgebras, see e.g.~\cite{KRT}). Both Hopf pairings on
$U_q(\widehat{\ssl}_n)$ and $U_q^\drj(\widehat{\ssl}_n)$ are known to be nondegenerate for $q$ not a root of unity.
Since $h,v$ respect the pairings, their injectivity follows.
\end{rem}


\subsection{Miki's isomorphism}
$\ $

 We recall the beautiful result of K.~Miki which provides an isomorphism
$^{'}\ddot{U}_{q,d}(\ssl_n)\iso \ddot{U}^{'}_{q,d}(\ssl_n)$
intertwining the \emph{vertical} and \emph{horizontal} embeddings of
quantum affine algebras of $\ssl_n$.

 To formulate the main result of this section, we need some more notation.

\noindent
$\circ$ Let $U_q(L\ssl_n)$ be obtained from $U_q(\widehat{\ssl}_n)$ by ``ignoring'' the generator $\wt{D}^{\pm 1}$
and taking a quotient by the ideal $(C-1)$, i.e., setting $C=1$.
The algebra $U_q(L\ssl_n)$ is usually called the quantum loop algebra of $\ssl_n$.
Analogously to $h$ and $v$, we have the following inclusions:
  $$^{'}\ddot{U}_{q,d}(\ssl_n)\overset{{'}h}\hookleftarrow U_q(\widehat{\ssl}_n) \overset{v{'}}\hookrightarrow \ddot{U}^{'}_{q,d}(\ssl_n)$$
and
  $$^{'}\ddot{U}_{q,d}(\ssl_n)\overset{{'}v}\hookleftarrow U_q(L\ssl_n) \overset{h{'}}\hookrightarrow  \ddot{U}^{'}_{q,d}(\ssl_n).$$

\noindent
$\circ$ Let $\sigma$ be the antiautomorphism of $U_q(\widehat{\ssl}_n)$ determined by
  $$\sigma\colon x^\pm_i\mapsto x^\pm_i,\ t_i\mapsto t_i^{-1},\ D\mapsto D^{-1}.$$

\noindent
$\circ$ Let $\eta$ be the antiautomorphism of $U_q(\widehat{\ssl}_n)$ determined by
  $$\eta\colon e_{i,k}\mapsto e_{i,-k},\ f_{i,k}\mapsto f_{i,-k},\ h_{i,l}\mapsto -C^lh_{i,-l},\ \psi_{i,0}\mapsto \psi_{i,0}^{-1},\
    C\mapsto C,\ \wt{D}\mapsto  \wt{D}\cdot \prod_{i=1}^{n-1}\psi_{i,0}^{-i(n-i)}.$$

\noindent
$\circ$ Let $^{'}Q$ be the automorphism of $^{'}\ddot{U}_{q,d}(\ssl_n)$ determined by
  $$^{'}Q\colon e_{i,k}\mapsto (-d)^ke_{i+1,k},\ f_{i,k}\mapsto (-d)^kf_{i+1,k},\ h_{i,l}\mapsto (-d)^lh_{i+1,l},\
   \psi_{i,0}\mapsto \psi_{i+1,0},\ q^{d_2}\mapsto q^{d_2}.$$

\noindent
$\circ$ Let $Q^{'}$ be the automorphism of $\ddot{U}^{'}_{q,d}(\ssl_n)$ such that it maps the generators other than
$\gamma^{\pm 1/2},q^{\pm d_1}$ as $^{'}Q$, while
  $$Q^{'}\colon \gamma^{1/2}\mapsto \gamma^{1/2},\ q^{d_1}\mapsto q^{d_1}\cdot \gamma^{-1}.$$

\noindent
$\circ$ Let $^{'}\Y_j\ (1\leq j\leq n)$ be the automorphism of $^{'}\ddot{U}_{q,d}(\ssl_n)$ determined by
  $$^{'}\Y_j\colon h_{i,l}\mapsto h_{i,l},\ \psi_{i,0}\mapsto \psi_{i,0},\ q^{d_2}\mapsto q^{d_2},$$
  $$e_{i,k}\mapsto (-d)^{-n\delta_{i,0}\delta_{j,n}-i\bar{\delta}_{i,j}+i\delta_{i,j-1}}e_{i,k-\bar{\delta}_{i,j}+\delta_{i,j-1}},\
    f_{i,k}\mapsto (-d)^{n\delta_{i,0}\delta_{j,n}+i\bar{\delta}_{i,j}-i\delta_{i,j-1}}f_{i,k+\bar{\delta}_{i,j}-\delta_{i,j-1}},$$
where
  $\bar{\delta}_{i,j}=
   \begin{cases}
     1 & \text{if}\ \ j \equiv i\ (\mathrm{mod}\ n) \\
     0 &  \text{otherwise}
   \end{cases}.$

\noindent
$\circ$ Let $\Y^{'}_j\ (1\leq j\leq n)$ be the automorphism of $\ddot{U}^{'}_{q,d}(\ssl_n)$ such that it maps the generators other than
$\psi^{\pm 1}_{i,0},\gamma^{\pm 1/2},q^{\pm d_1}$ as $^{'}\Y_j$, while
  $$\Y_j^{'}\colon \gamma^{1/2}\mapsto \gamma^{1/2},\
              \psi_{i,0}\mapsto \gamma^{-\bar{\delta}_{i,j}+\delta_{i,j-1}}\psi_{i,0},$$
  $$\Y_j^{'}\colon q^{d_1}\mapsto q^{d_1}\cdot \gamma^{-\frac{n+1}{2n}}\cdot K_j\
              \mathrm{with}\ K_j=\prod_{l=1}^{j-1} q^{\frac{l}{n}h_{l,0}} \prod_{l=j}^{n-1} q^{\frac{l-n}{n}h_{l,0}},$$
where we follow the conventions of Section 1.1 and add elements
$\gamma^{\frac{1}{2n}}, q^{\frac{h_{j,0}}{2n}}$ to $\ddot{U}^{'}_{q,d}(\ssl_n)$.

\begin{thm}\cite[Proposition 1]{M2}\label{Miki Theorem 1}
There exists an algebra isomorphism
  $$\varpi\colon\ ^{'}\ddot{U}_{q,d}(\ssl_n)\iso \ddot{U}^{'}_{q,d}(\ssl_n)$$
satisfying the following properties:
  $$\varpi\circ {{'}h}=v{'},\ \varpi\circ {{'}v}\circ \eta\circ \sigma=h{'},\
    Q^{'}\circ \Y_n^{'}\circ \varpi=\varpi\circ {^{'}\Y_1^{-1}}\circ {^{'}Q}.$$
\end{thm}

\begin{rem}
(a) Let $U^{\mathrm{tor}}_{q,d}(\ssl_n)$ be obtained from $\ddot{U}_{q,d}(\ssl_n)$
by ``ignoring'' the generators $q^{\pm d_1}$ and $q^{\pm d_2}$.
The construction of $\varpi$ in~\cite{M2} was based on the previous work~\cite{M},
where an automorphism $\overline{\varpi}$ of $U^{\mathrm{tor}}_{q,d}(\ssl_n)$ was established.

\noindent
(b) The generators $x^{\pm}_{i,k},  h_{i,l}, k_i^{\pm 1}, C^{\pm 1}, D^{\pm 1}, \wt{D}^{\pm 1}$ from~\cite{M2}
are related to our generators via
  $$x^+_{i,k}\leftrightarrow d^{ik}e_{i,k},\ x^-_{i,k}\leftrightarrow d^{ik}f_{i,k},\ h_{i,l} \leftrightarrow d^{il}\gamma^{l/2}h_{i,l},$$
  $$k^{\pm 1}_i\leftrightarrow \psi^{\pm 1}_{i,0},\ C^{\pm 1}\leftrightarrow \gamma^{\pm 1},\ D^{\pm 1}\leftrightarrow q^{\pm d_2},\
    \wt{D}^{\pm 1}\leftrightarrow q^{\mp nd_1}\cdot q^{\pm \sum_{j=1}^{n-1} \frac{j(n-j)}{2}h_{j,0}},$$
while the parameters $q,\xi$ from~\cite{M2} are related to our parameters $q,d$ via
  $$q \leftrightarrow q,\ \xi \leftrightarrow d^{-n}.$$

\noindent
(c) The aforementioned choice of generators from~\cite{M2} is
convenient as there is no need to add elements
$\left\{q^{\frac{h_{j,0}}{N}}, q^{\frac{c}{N}}, q^{\frac{d_1}{N}}, q^{\frac{d_2}{N}}\right\}_{N\in \ZZ_{>0}}$.
However, we prefer the current presentation as it is more symmetric and suitable for the rest of our exposition.
\end{rem}

 We conclude this subsection by computing images of some generators of $^{'}\ddot{U}_{q,d}(\ssl_n)$ under $\varpi$.

\begin{prop}\label{explicit formulas for Miki}
(a) We have
  $$\varpi\colon e_{i,0}\mapsto e_{i,0},\ f_{i,0}\mapsto f_{i,0},\ \psi^{\pm 1}_{i,0}\mapsto \psi^{\pm 1}_{i,0}\ \mathrm{for}\ i\in [n]^\times,$$
  $$\varpi\colon \psi_{0,0}^{\pm 1}\mapsto \gamma^{\pm 1} \cdot \psi_{0,0}^{\pm 1},\
    q^{\pm d_2}\mapsto q^{\mp nd_1}\cdot q^{\pm \sum_{j=1}^{n-1} \frac{j(n-j)}{2}h_{j,0}},$$
  $$\varpi\colon e_{0,0}\mapsto d\cdot \gamma\psi_{0,0}\cdot [\cdots[f_{1,1},f_{2,0}]_q,\cdots,f_{n-1,0}]_q,$$
  $$\varpi\colon f_{0,0}\mapsto d^{-1} \cdot [e_{n-1,0},\cdots,[e_{2,0},e_{1,-1}]_{q^{-1}}\cdots]_{q^{-1}}\cdot \psi_{0,0}^{-1}\gamma^{-1}.$$

\noindent
(b) For $i\in [n]^\times$, we have
  $$\varpi(h_{i,1})=
    (-1)^{i+1}d^{-i}\cdot [[\cdots[[\cdots[f_{0,0},f_{n-1,0}]_q,\cdots,f_{i+1,0}]_q,f_{1,0}]_q,\cdots ,f_{i-1,0}]_q,f_{i,0}]_{q^2},$$
  $$\varpi(h_{i,-1})=
    (-1)^{i+1}d^i\cdot [e_{i,0},[\cdots,[e_{1,0},[e_{i+1,0},\cdots,[e_{n-1,0},e_{0,0}]_{q^{-1}}\cdots]_{q^{-1}}]_{q^{-1}}\cdots]_{q^{-1}}]_{q^{-2}}.$$

\noindent
(c) For $i=0$, we have
  $$\varpi(h_{0,1})=(-1)^nd^{1-n}\cdot [[\cdots[f_{1,1},f_{2,0}]_q,\cdots,f_{n-1,0}]_q,f_{0,-1}]_{q^2},$$
  $$\varpi(h_{0,-1})=(-1)^nd^{n-1}\cdot [e_{0,1},[e_{n-1,0},\cdots,[e_{2,0},e_{1,-1}]_{q^{-1}}\cdots]_{q^{-1}}]_{q^{-2}}.$$

\noindent
(d) We have
  $$\varpi(e_{0,-1})=(-d)^{n}e_{0,1},\ \varpi(f_{0,1})=(-d)^{-n}f_{0,-1}.$$
\end{prop}

\begin{proof}[Proof of Proposition~\ref{explicit formulas for Miki}]
$\ $

(a) Follow straightforwardly by applying the equality $\varpi\circ {{'}h}=v{'}$
to the explicit formulas for $\Phi(x^\pm_i), \Phi(t_i), \Phi(D)$ with
$\Phi\colon U_q^\drj(\widehat{\ssl}_n)\iso U_q(\widehat{\ssl}_n)$ from subsection~\ref{two copies}.

(b) We will need the following formulas expressing $h_{i,\pm 1}$ in the Drinfeld-Jimbo presentation:
\begin{equation}\label{hi1-dj}
  \Phi^{-1}(h_{i,1})=(-1)^i[[\cdots [[\cdots[x_0^+,x^+_{n-1}]_{q^{-1}},\cdots, x^+_{i+1}]_{q^{-1}}, x^+_1]_{q^{-1}},\cdots,x^+_{i-1}]_{q^{-1}},x^+_i]_{q^{-2}},
\end{equation}
\begin{equation}\label{hi-1-dj}
  \Phi^{-1}(h_{i,-1})=(-1)^i[x^-_i,\cdots,[x^-_1,[x^-_{i+1},\cdots,[x^-_{n-1},x^-_0]_q\cdots ]_q]_q\cdots ]_{q^2}.
\end{equation}
Formulas~(\ref{hi1-dj},\ref{hi-1-dj}) are proved by applying iteratively two useful identities
involving $q$-brackets:
   $$[a,[b,c]_u]_v=[[a,b]_x,c]_{uv/x}+x\cdot [b,[a,c]_{v/x}]_{u/x},$$
   $$[[a,b]_u,c]_v=[a,[b,c]_x]_{uv/x}+x\cdot [[a,c]_{v/x},b]_{u/x},$$
compare to our proof of Theorem~\ref{main5}
(we leave verification of details to the interested reader).

Applying the equality $\varpi\circ {{'}v}\circ \eta=h{'}\circ \sigma^{-1}$ to formulas~(\ref{hi1-dj}) and~(\ref{hi-1-dj}),
we get the claimed formulas for $\varpi(h_{i,-1})$ and $\varpi(h_{i,1})$, respectively.

(c) Applying the equality $Q^{'}\circ \Y_n^{'}\circ \varpi=\varpi\circ {^{'}\Y_1^{-1}}\circ {^{'}Q}$ to $h_{n-1,\pm 1}$
and using the formulas for $\varpi(h_{n-1,\pm 1})$ from (b), we obtain the formulas for $\varpi(h_{0,\pm 1})$.

(d) It suffices to apply the equality $Q^{'}\circ \Y_n^{'}\circ \varpi=\varpi\circ {^{'}\Y_1^{-1}}\circ {^{'}Q}$ to $e_{n-1,0}$ and $f_{n-1,0}$.
\end{proof}


\subsection{Fock and Macmahon modules}\label{section Fock}
$\ $

In this section, we recall two interesting classes of $^{'}\ddot{U}_{q,d}(\ssl_n)$-modules constructed in~\cite{FJMM1}.
They depend on two parameters: $0\leq p\leq n-1$ and $u\in \CC^\times$. We also set
\begin{equation}\tag{$\diamond$}
   q_1:=q^{-1}d,\ q_2:=q^2,\ q_3:=q^{-1}d^{-1}\ \mathrm{and}\ \phi(t):=\frac{q^{-1}t-q}{t-1}.
\end{equation}

\underline{Assumption:} In the rest of this paper, we assume that $q_1,q_2,q_3$ are \emph{generic}, that is,
\begin{equation}\tag{G}\label{generic}
q_1^aq_2^bq_3^c=1\ \mathrm{for\ some}\ a,b,c\in \ZZ\ \mathrm{implies}\ a=b=c.
\end{equation}

 Given $\nu\in \CC^\times$ and a collection of formal series $\boldsymbol{\phi}(z)=\{\phi^\pm_i(z)\}_{i\in [n]},\ \phi^\pm_i(z)\in \CC[[z^{\mp 1}]]$,
a vector $v$ of an $^{'}\ddot{U}_{q,d}(\ssl_n)$-module $V$ is
said to have \emph{weight} $(\nu;\boldsymbol{\phi}(z))$ if $q^{d_2}v=\nu\cdot v$ and
$\psi^\pm_i(z)v=\phi^\pm_i(z)\cdot v$ for any $i\in [n]$.
The module $V$ is called a \emph{lowest weight module}
if it is generated by a weight vector $v$ such that ${^{'}\ddot{U}^{-}}v=0$.
Such $v$ is called a \emph{lowest weight vector}, and its weight the \emph{lowest weight} of $V$.
Given $\nu\in \CC^\times$ and $\boldsymbol{\phi}(z)$ with $\phi^+_i(\infty)\phi^-_i(0)=1$ for every $i\in [n]$,
there is a unique irreducible lowest weight module of that lowest weight.
If $\phi^\pm_i(z)$ are expansions of a rational function
$\phi_i(z)$ at $z=0,\infty$, then we write $(\nu;\boldsymbol{\phi}(z))=(\nu;\phi_i(z))_{i\in [n]}$.

\medskip
\noindent
$\bullet$ \emph{Fock modules $F^{(p)}(u)$.}

 The most basic lowest weight $^{'}\ddot{U}_{q,d}(\ssl_n)$-modules are the Fock modules $F^{(p)}(u)$
with the basis $\{|\lambda\rangle\}$ labeled by all partitions $\lambda$.
Given such a partition $\lambda=(\lambda_1,\lambda_2,\ldots)$, we define
$\lambda+1_l:=(\lambda_1,\ldots\lambda_l+1,\ldots)$, $|\lambda|:=\sum_l\lambda_l$, and
we denote the transposed partition by $\lambda'=(\lambda'_1,\lambda'_2,\ldots)$. We also write $a\equiv b$ if $a-b$ is divisible by $n$.

\begin{prop}\label{Fock module}~\cite[Proposition 3.3]{FJMM1}
The $^{'}\ddot{U}_{q,d}(\ssl_n)$-action on $F^{(p)}(u)$ is given by the following formulas:
  $$\langle \lambda+1_l| e_i(z) |\lambda\rangle=
    \bar{\delta}_{p+l-\lambda_l,i+1}\prod_{1\leq s<l}^{p+s-\lambda_s\equiv i} \phi(q_1^{\lambda_s-\lambda_l-1}q_3^{s-l})
    \prod_{1\leq s<l}^{p+s-\lambda_s\equiv i+1} \phi(q_1^{\lambda_l-\lambda_s}q_3^{l-s}) \delta(q_1^{\lambda_l}q_3^{l-1}u/z),$$
  $$\langle \lambda| f_i(z) |\lambda+1_l\rangle=
    \bar{\delta}_{p+l-\lambda_l,i+1}\prod_{s>l}^{p+s-\lambda_s\equiv i} \phi(q_1^{\lambda_s-\lambda_l-1}q_3^{s-l})
    \prod_{s>l}^{p+s-\lambda_s\equiv i+1} \phi(q_1^{\lambda_l-\lambda_s}q_3^{l-s}) \delta(q_1^{\lambda_l}q_3^{l-1}u/z),$$
  $$\langle \lambda| \psi^\pm_i(z) |\lambda\rangle=
    \prod_{s\geq 1}^{p+s-\lambda_s\equiv i} \phi(q_1^{\lambda_s-1}q_3^{s-1}u/z)
    \prod_{s\geq 1}^{p+s-\lambda_s\equiv i+1} \phi(q_1^{\lambda_s-1}q_3^{s-2}u/z)^{-1},\
    \langle \lambda| q^{d_2}|\lambda\rangle=q^{|\lambda|},$$
while all other matrix coefficients are zero.
$F^{(p)}(u)$ is an irreducible lowest weight module of the lowest weight $(1;\phi(z/u)^{\delta_{i,p}})_{i\in [n]}$
and with $|\emptyset\rangle$ being the corresponding lowest weight vector.
\end{prop}

\begin{defn}
 For $\bar{c}\in (\CC^\times)^{[n]}$, let $\tau^p_{u,\bar{c}}$ be the twist of this representation
by the automorphism $\chi_{p,\bar{c}}$ of ${^{'}\ddot{U}}_{q,d}(\ssl_n)$ defined via
  $e_{i,k}\mapsto c_ie_{i,k},\ f_{i,k}\mapsto c_i^{-1}f_{i,k},\ \psi_{i,k}\mapsto \psi_{i,k},\
   q^{d_2}\mapsto q^{-\frac{p(n-p)}{2}}\cdot q^{d_2}$.
\end{defn}

 Given a collection $\{(p_k,u_k,\bar{c}_k)\}_{k=1}^r$ with
$0\leq p_k\leq n-1, u_k\in \CC^\times, \bar{c}_k\in (\CC^\times)^{[n]}$,
we call it \emph{generic} if for any pair $1\leq s'<s\leq r$,
there are no $a,b\in \ZZ$ such that $b-a\equiv p_{s'}-p_s$ and $u_s=u_{s'}q_1^{-a}q_3^{-b}$.
We have the following simple result (see~\cite[Lemma 4.1]{FJMM1}).

\begin{lem}\label{tensor}
For a generic collection $\{(p_k,u_k,\bar{c}_k)\}_{k=1}^r$, the
coproduct $\Delta$ of~(\ref{coproduct}) endows
$\tau^{p_1}_{u_1,\bar{c}_1}\otimes \cdots\otimes \tau^{p_r}_{u_r,\bar{c}_r}$
with a structure of an $^{'}\ddot{U}_{q,d}(\ssl_n)$-module.
It is an irreducible lowest weight module generated by the lowest weight vector
$|\emptyset\rangle\otimes \cdots\otimes |\emptyset\rangle$.
\end{lem}

\begin{rem}
To see the irreducibility, one checks that the action of
commuting series $\psi^\pm_i(z)$ is diagonalizable and has a simple
joint spectrum (here we use $q_1,q_2,q_3$ being \emph{generic}, see~(\ref{generic})).
\end{rem}

\medskip
\noindent
$\bullet$ \emph{Macmahon modules $M^{(p)}(u,K)$.}

 For $K\in \CC^\times$, set $\phi^K(t):=\frac{K^{-1}t-K}{t-1}$.
We call $K$ \emph{generic} if $K\notin q^\ZZ d^\ZZ$.
For such $K$, the unique irreducible lowest weight $^{'}\ddot{U}_{q,d}(\ssl_n)$-module of the
lowest weight $(1;\phi^K(z/u)^{\delta_{i,p}})_{i\in [n]}$ is called the Macmahon module, denoted by $M^{(p)}(u,K)$.
They were first studied in~\cite{FJMM1}.
Recall that a collection of partitions $\bar{\lambda}=\{\lambda^{(r)}\}_{r\in \ZZ_{>0}}$ is called a \emph{plane partition} if
  $$\lambda^{(r)}_l\geq \lambda^{(r+1)}_l\ \mathrm{for\ all}\ r,l\in \ZZ_{>0}\ \  \mathrm{and}\ \ \lambda^{(r)}=\emptyset \ \mathrm{for}\ r\gg 0.$$

\begin{prop}\cite[Theorem 4.3]{FJMM1}
 For a generic $K$, the vector space $M^{(p)}(u,K)$ has a basis $\{|\bar{\lambda}\rangle\}$
(labeled by all plane partitions) with $|\bar{\emptyset}\rangle$ being its lowest weight vector.
\end{prop}

In this paper, we will not need explicit formulas for the $^{'}\ddot{U}_{q,d}(\ssl_n)$-action in the basis $\{|\bar{\lambda}\rangle\}$.


\subsection{Vertex representations}
$\ $

 In this section, we recall a family of vertex $\ddot{U}^{'}_{q,d}(\ssl_n)$-representations from~\cite{S}
generalizing the construction of~\cite{FJ} for quantum affine algebras.
Let $S_n$ be the \emph{generalized} Heisenberg algebra generated by $\{H_{i,k}|i\in [n],k\in \ZZ\backslash\{0\}\}$ and
a central element $H_0$ with the defining relations
  $$[H_{i,k}, H_{j,l}]=d^{-km_{i,j}}\frac{[k]_q\cdot [ka_{i,j}]_q}{k}\delta_{k,-l}\cdot H_0.$$
Let $S_{n}^+$ be the subalgebra of $S_n$ generated by $\{H_{i,k}|i\in[n],k>0\}\sqcup \{H_0\}$,
and let $\CC v_0$ be the $S_n^+$-representation with $H_{i,k}$ acting trivially and $H_0$ acting via the identity operator.
The induced representation $F_n:=\mathrm{Ind}_{S_{n}^+}^{S_n}  \CC v_0$ is called the \emph{Fock representation} of $S_n$.

 We denote by $\{\bar{\alpha}_i\}_{i=1}^{n-1}$ the simple roots of $\ssl_n$,
by $\{\bar{\Lambda}_i\}_{i=1}^{n-1}$ the fundamental weights of $\ssl_n$,
by $\{\bar{h}_i\}_{i=1}^{n-1}$ the simple coroots of $\ssl_n$.
Let $\bar{Q}:=\bigoplus_{i=1}^{n-1} \ZZ\bar{\alpha}_i$ be the root lattice of $\ssl_n$,
$\bar{P}:=\bigoplus_{i=1}^{n-1} \ZZ\bar{\Lambda}_i=\bigoplus_{i=2}^{n-1} \ZZ{\bar{\alpha}_i}\oplus \ZZ\bar{\Lambda}_{n-1}$
be the weight lattice of $\ssl_n$.
We also set
  $$\bar{\alpha}_0:=-\sum_{i=1}^{n-1}\bar{\alpha}_i\in \bar{Q},\
    \bar{\Lambda}_0:=0\in \bar{P},\
    \bar{h}_0:=-\sum_{i=1}^{n-1}\bar{h}_i.$$

 Let $\CC\{\bar{P}\}$ be the $\CC$-algebra generated by
$e^{\bar{\alpha}_2},\ldots,e^{\bar{\alpha}_{n-1}},e^{\bar{\Lambda}_{n-1}}$
with the defining relations:
  $$e^{\bar{\alpha}_i}\cdot e^{\bar{\alpha}_j}=(-1)^{\langle \bar{h}_i,\bar{\alpha}_j \rangle}e^{\bar{\alpha}_j}\cdot e^{\bar{\alpha}_i},\
    e^{\bar{\alpha}_i}\cdot e^{\bar{\Lambda}_{n-1}}=(-1)^{\delta_{i,n-1}}e^{\bar{\Lambda}_{n-1}}\cdot e^{\bar{\alpha}_i}.$$
For $\alpha=\sum_{i=2}^{n-1} m_i\bar{\alpha}_i+m_n\bar{\Lambda}_{n-1}$, we define $e^{\bar{\alpha}}\in \CC\{\bar{P}\}$ via
 $$e^{\bar{\alpha}}:=(e^{\bar{\alpha}_2})^{m_2}\cdots (e^{\bar{\alpha}_{n-1}})^{m_{n-1}}(e^{\bar{\Lambda}_{n-1}})^{m_n}.$$
Let $\CC\{\bar{Q}\}$ be the subalgebra of $\CC\{\bar{P}\}$ generated by $\{e^{\bar{\alpha}_i}\}_{i=1}^{n-1}$.

For every $0\leq p\leq n-1$, define the space
  $$W(p)_n:=F_n\otimes \CC\{\bar{Q}\}e^{\bar{\Lambda}_p}.$$
Consider the operators $H_{i,l}, e^{\bar{\alpha}}, \partial_{\bar{\alpha}_i}, z^{H_{i,0}}, \mathrm{d}$ acting on $W(p)_n$,
which assign to every element
  $$v\otimes e^{\bar{\beta}}=(H_{i_1,-k_1}\cdots H_{i_N,-k_N} v_0)\otimes e^{\sum_{j=1}^{n-1}m_j\bar{\alpha}_j+\bar{\Lambda}_p}\in W(p)_n$$
the following values:
\begin{equation*}
  H_{i,l}(v\otimes e^{\bar{\beta}}):=(H_{i,l}v)\otimes e^{\bar{\beta}},\
  e^{\bar{\alpha}}(v\otimes e^{\bar{\beta}}):=v\otimes e^{\bar{\alpha}}e^{\bar{\beta}},\
  \partial_{\bar{\alpha}_i}(v\otimes e^{\bar{\beta}}):=\langle \bar{h}_i,\bar{\beta} \rangle v\otimes e^{\bar{\beta}},
\end{equation*}
\begin{equation*}
  z^{H_{i,0}}(v\otimes e^{\bar{\beta}}):=
  z^{\langle \bar{h}_i,\bar{\beta} \rangle} d^{\frac{1}{2}\sum_{j=1}^{n-1}\langle \bar{h}_i,m_j\bar{\alpha}_j \rangle m_{i,j}} v\otimes e^{\bar{\beta}},
\end{equation*}
\begin{equation*}
  \mathrm{d}(v\otimes e^{\bar{\beta}}):=(-\sum k_i+((\bar{\Lambda}_p,\bar{\Lambda}_p)-(\bar{\beta},\bar{\beta}))/2) v\otimes e^{\bar{\beta}}.
\end{equation*}

The following result provides a natural structure of an $\ddot{U}^{'}_{q,d}(\ssl_n)$-module on $W(p)_n$.

\begin{prop}\cite[Proposition 3.2.2]{S}\label{Saito representation}
 For any $\bar{c}=(c_0,\ldots, c_{n-1})\in (\CC^\times)^{[n]},\ u\in \CC^\times$, and $0\leq p\leq n-1$,
the following formulas define an action of $\ddot{U}^{'}_{q,d}(\ssl_n)$ on $W(p)_n$:
\begin{equation*}
  \rho^p_{u,\bar{c}}(e_i(z))=c_i \exp\left( \sum_{k>0} \frac{q^{-k/2}u^{-k}}{[k]_q}H_{i,-k}z^k\right)
  \exp\left( -\sum_{k>0}\frac{q^{-k/2}u^k}{[k]_q}H_{i,k}z^{-k} \right) e^{\bar{\alpha}_i}\left(\frac{z}{u}\right)^{1+H_{i,0}},
\end{equation*}
\begin{equation*}
  \rho^p_{u,\bar{c}}(f_i(z))=\frac{(-1)^{n\delta_{i,0}}}{c_i} \exp\left( -\sum_{k>0} \frac{q^{k/2}u^{-k}}{[k]_q}H_{i,-k}z^k\right)
  \exp\left( \sum_{k>0}\frac{q^{k/2}u^k}{[k]_q}H_{i,k}z^{-k} \right) e^{-\bar{\alpha}_i}\left(\frac{z}{u}\right)^{1-H_{i,0}},
\end{equation*}
\begin{equation*}
  \rho^p_{u,\bar{c}}(\psi_i^{\pm}(z))=\exp\left( \pm(q-q^{-1})\sum_{k>0} H_{i,\pm k}(z/u)^{\mp k} \right) q^{\pm \partial_{\bar{\alpha}_i}},\
  \rho^p_{u,\bar{c}}(\gamma^{1/2})=q^{1/2},\ \rho^p_{u,\bar{c}}(q^{d_1})=q^{\mathrm{d}}.
\end{equation*}
  $W(p)_n$ is an irreducible $\ddot{U}^{'}_{q,d}(\ssl_n)$-module if $q,qd,qd^{-1}$ are not roots of unity.
\end{prop}

\begin{rem}
(a) The irreducibility of $\rho^p_{u,\bar{c}}$ follows from the irreducibility of
the $S_n$-module $F_n$ and level one vertex $U_q(\widehat{\ssl}_n)$-modules of~\cite{FJ},
established at~\cite{CJ}.

\noindent
(b) The factor $(-1)^{n\delta_{i,0}}$ in $\rho^p_{u,\bar{c}}(f_i(z))$
(missing in~\cite{S,FT1}) is due to $(e^{\bar{\alpha}_i})^{-1}=(-1)^{n\delta_{i,0}}e^{-\bar{\alpha}_i}$.
\end{rem}

\subsection{Shuffle algebra}
$\ $

 Consider an $\NN^{[n]}$-graded $\CC$-vector space
  $$\sS=\underset{\overline{k}=(k_0,\ldots,k_{n-1})\in \NN^{[n]}}\bigoplus\sS_{\overline{k}},$$
where $\sS_{(k_0,\ldots,k_{n-1})}$ consists of $\prod \mathfrak{S}_{k_i}$-symmetric rational functions in the variables
$\{x_{i,r}\}_{i\in [n]}^{1\leq r\leq k_i}$.
We also fix an $n\times n$ matrix of rational functions
$\Omega=(\omega_{i,j}(z))_{i,j\in [n]} \in \mathrm{Mat}_{n\times n}(\CC(z))$
by setting
  $$\omega_{i,i}(z)=\frac{z-q^{-2}}{z-1},\ \omega_{i,i+1}(z)=\frac{d^{-1}z-q}{z-1},\
    \omega_{i,i-1}(z)=\frac{z-qd^{-1}}{z-1},\ \mathrm{and}\ \omega_{i,j}(z)=1\ \mathrm{otherwise}.$$
Let us now introduce the bilinear $\star$ product on $\sS$: given  $F\in \sS_{\overline{k}}, G\in \sS_{\overline{l}}$,
define $F\star G\in \sS_{\overline{k}+\overline{l}}$ by
  $$(F\star G)(x_{0,1},\ldots,x_{0,k_0+l_0};\ldots;x_{n-1,1},\ldots, x_{n-1,k_{n-1}+l_{n-1}}):=$$
  $$\Sym_{\prod\mathfrak{S}_{k_i+l_i}}
    \left(F(\{x_{i,r}\}_{i\in [n]}^{1\leq r\leq k_i})
    G(\{x_{i',r'}\}_{i'\in [n]}^{k_{i'}<r'\leq k_{i'}+l_{i'}})\cdot
    \prod_{i\in [n]}^{i'\in [n]}\prod_{r\leq k_i}^{r'>k_{i'}}\omega_{i,i'}(x_{i,r}/x_{i',r'})\right).$$
Here and afterwards, given a function $f\in \CC(\{x_{i,1},\ldots,x_{i,m_i}\}_{i\in [n]})$, we define
  $$\Sym_{\prod\mathfrak{S}_{m_i}}(f):=\prod_{i\in [n]}\frac{1}{m_i!}\cdot
    \sum_{(\sigma_0,\ldots,\sigma_{n-1})\in \mathfrak{S}_{m_0}\times \ldots\times \mathfrak{S}_{m_{n-1}}}
    f(\{x_{i,\sigma_i(1)},\ldots,x_{i,\sigma_i(m_i)}\}_{i\in [n]}).$$

\medskip
 This endows $\sS$ with a structure of an associative unital algebra with the unit $\textbf{1}\in \sS_{(0,\ldots,0)}$.
We will be interested only in a certain subspace of $\sS$, defined by the \emph{pole} and \emph{wheel conditions}:

\medskip
\noindent
$\bullet$ We say that $F\in \sS_{\overline{k}}$ satisfies the \emph{pole conditions} if and only if
  $$F=\frac{f(x_{0,1},\ldots,x_{n-1,k_{n-1}})}{\prod_{i\in [n]}\prod_{r\leq k_i}^{r'\leq k_{i+1}}(x_{i,r}-x_{i+1,r'})},\
    \mathrm{where}\ f\in (\CC[x_{i,r}^{\pm 1}]_{i\in [n]}^{1\leq r\leq k_i})^{\prod \mathfrak{S}_{k_i}}.$$

\noindent
$\bullet$ We say that $F\in \sS_{\overline{k}}$ satisfies the \emph{wheel conditions} if and only if
  $$F(\{x_{i,r}\})=0\ \mathrm{once}\ x_{i,r_1}/x_{i+\epsilon,l}=qd^{\epsilon}\ \mathrm{and}\ x_{i+\epsilon,l}/x_{i,r_2}=qd^{-\epsilon}\
    \mathrm{for\ some}\ \epsilon, i, r_1, r_2, l,$$
where $\epsilon\in \{\pm 1\}, i\in [n], 1\leq r_1,r_2\leq k_i, 1\leq l\leq k_{i+\epsilon}$
and we use the cyclic notation
$x_{n,l}:=x_{0,l}, k_n:=k_0, x_{-1,l}:=x_{n-1,l}, k_{-1}:=k_{n-1}$ as before.

\medskip
\noindent Let $S_{\overline{k}}\subset \sS_{\overline{k}}$ be the subspace of all elements $F$ satisfying the above two conditions and set
  $$S:=\underset{\overline{k}\in \NN^{[n]}}\bigoplus S_{\overline{k}}.$$
Further $S_{\overline{k}}=\oplus_{r\in \ZZ}S_{\overline{k},r}$ with
$S_{\overline{k},r}:=\{F\in S_{\overline{k}}|\mathrm{tot.deg}(F)=r\}$.
The following is straightforward.

\begin{lem}
 The subspace $S\subset\sS$ is $\star$-closed.
\end{lem}

 Now we are ready to introduce one of the key actors of this paper:

\begin{defn}
 The algebra $(S,\star)$ is called the shuffle algebra (of $A_{n-1}^{(1)}$-type).
\end{defn}

\medskip

 Recall the subalgebra $\ddot{U}^+$ of $\ddot{U}_{q,d}(\ssl_n)$ from Section 1.2.
By standard results\footnote{\ See~\cite[Theorem 4.2.1]{Jan} for the case of finite quantum groups,~\cite[Theorem 3.2]{H} for the case
of quantum affine algebras, and~\cite[Proposition 1.3]{T} for the case of quantum toroidal of $\gl_1$.}, $\ddot{U}^+$
is generated by $\{e_{i,k}\}_{i\in [n]}^{k\in \ZZ}$ with the defining relations (T2, T7.1).
We equip the algebra $\ddot{U}^+$ with the $\NN^{[n]}\times \ZZ$--grading by assigning $\deg(e_{i,k})=(1_i;k)$,
where $1_i\in \NN^{[n]}$ is the vector with the $i$-th coordinate $1$ and all other coordinates being zero.

 The following result is straightforward:

\begin{prop}\label{homomorphisms}
 There exists a unique algebra homomorphism $\Psi\colon \ddot{U}^+\to \sS$ such that
$\Psi(e_{i,k})=x_{i,1}^k$ for any $i\in [n],k\in \ZZ$. In particular, $\mathrm{Im}(\Psi)\subset S$.
\end{prop}

 The following beautiful result was recently proved by A.~Negut:

\begin{thm}\cite[Theorem 1.1]{N}\label{Negut theorem}
 The homomorphism $\Psi\colon \ddot{U}^+\to S$ is an isomorphism of $\NN^{[n]}\times \ZZ$-graded algebras.
\end{thm}


\subsection{Shuffle bimodules}
$\ $

 Following the ideas of~\cite{FJMM2}, we introduce three families of $S$-bimodules.

\noindent
$\bullet$ \emph{Shuffle modules $S_{1,p}(u)$.}

 For $u\in \CC^\times$ and $0\leq p\leq n-1$, consider an $\NN^{[n]}$-graded $\CC$-vector space
  $$S_{1,p}(u)=\underset{\overline{k}=(k_0,\ldots,k_{n-1})\in \NN^{[n]}}\bigoplus S_{1,p}(u)_{\overline{k}},$$
where the degree $\overline{k}$ component $S_{1,p}(u)_{\overline{k}}$ consists of $\prod \mathfrak{S}_{k_i}$-symmetric rational
functions $F$ in the variables $\{x_{i,r}\}_{i\in [n]}^{1\leq r\leq k_i}$
satisfying the following three conditions:

\noindent
(i) \emph{Pole conditions}, that is,
  $$F=\frac{f(x_{0,1},\ldots,x_{n-1,k_{n-1}})}{\prod_{i\in [n]}\prod_{r\leq k_i}^{r'\leq k_{i+1}}(x_{i,r}-x_{i+1,r'})\cdot \prod_{r=1}^{k_p} (x_{p,r}-u)},\
    \mathrm{where}\ f\in (\CC[x_{i,r}^{\pm 1}]_{i\in [n]}^{1\leq r\leq k_i})^{\prod \mathfrak{S}_{k_i}}.$$

\noindent
(ii) \emph{First kind wheel conditions}, that is,
  $$F(\{x_{i,r}\})=0\ \mathrm{once}\
    x_{i,r_1}/x_{i+\epsilon,l}=qd^{\epsilon}\ \mathrm{and}\ x_{i+\epsilon,l}/x_{i,r_2}=qd^{-\epsilon}\
    \mathrm{for\ some}\ \epsilon, i, r_1, r_2, l,$$
where $\epsilon\in \{\pm 1\}, i\in [n], 1\leq r_1,r_2\leq k_i, 1\leq l\leq k_{i+\epsilon}$ and we use the cyclic notation.

\noindent
(iii) \emph{Second kind wheel conditions}, that is,
  $$f(\{x_{i,r}\})=0\ \mathrm{once}\ x_{p,r_1}=u\ \mathrm{and}\ x_{p,r_2}=q^2u\ \mathrm{for\ some}\ 1\leq r_1,r_2\leq k_p,$$
where $f(\{x_{i,r}\}):=\prod_{r=1}^{k_p} (x_{p,r}-u)\cdot F(\{x_{i,r}\})$.

\medskip
Fix $\bar{c}\in (\CC^\times)^{[n]}$.
Given $F\in S_{\overline{k}}$ and $G\in S_{1,p}(u)_{\overline{l}}$, we define
$F\star G, G\star F\in S_{1,p}(u)_{\overline{k}+\overline{l}}$ by
\begin{multline}\label{left}
  (F\star G)(x_{0,1},\ldots,x_{0,k_0+l_0};\ldots;x_{n-1,1},\ldots,x_{n-1,k_{n-1}+l_{n-1}}):=\prod_{i\in [n]} c_i^{k_i}\times\\
  \Sym_{\prod\mathfrak{S}_{k_i+l_i}}\left(F(\{x_{i,r}\}_{i\in [n]}^{r\leq k_i})
  G(\{x_{i',r'}\}_{i'\in [n]}^{r'> k_{i'}})
  \prod_{i\in [n]}^{i'\in [n]}\prod_{r\leq k_i}^{r'>k_{i'}}\omega_{i,i'}(x_{i,r}/x_{i',r'}) \prod_{r=1}^{k_p} \phi(x_{p,r}/u)\right)
\end{multline}
and
\begin{multline}\label{right}
  (G\star F)(x_{0,1},\ldots,x_{0,k_0+l_0};\ldots;x_{n-1,1},\ldots,x_{n-1,k_{n-1}+l_{n-1}}):=\\
  \Sym_{\prod\mathfrak{S}_{k_i+l_i}}\left(G(\{x_{i,r}\}_{i\in [n]}^{r\leq l_i})
  F(\{x_{i',r'}\}_{i'\in [n]}^{r'>l_{i'}})
  \prod_{i\in [n]}^{i'\in [n]}\prod_{r\leq l_i}^{r'>l_{i'}}\omega_{i,i'}(x_{i,r}/x_{i',r'})\right).
\end{multline}
These formulas endow $S_{1,p}(u)$ with a structure of an $S$-bimodule.

Identifying $S$ with $^{'}\ddot{U}^+\simeq \ddot{U}^+$ via $\Psi$ (see Theorem~\ref{Negut theorem}),
we get two commuting $^{'}\ddot{U}^+$-actions on $S_{1,p}(u)$.
Our next result extends one of these to an action of the entire algebra $^{'}\ddot{U}_{q,d}(\ssl_n)$.

\begin{prop}\label{extension to toroidal action}
The following formulas define an action of $^{'}\ddot{U}_{q,d}(\ssl_n)$ on $S_{1,p}(u)$:
  $$\pi^p_{u,\overline{c}}(q^{d_2})G=q^{-\frac{p(n-p)}{2}+|\overline{k}|}\cdot G,\
    \pi^p_{u,\overline{c}}(e_{i,k})G=x_i^k\star G,$$
  $$\pi^p_{u,\overline{c}}(h_{i,0})G=(2k_i-k_{i-1}-k_{i+1}-\delta_{i,p})\cdot G,$$
  $$\pi^p_{u,\overline{c}}(h_{i,l})G=
    \left(\frac{1}{l}\sum_{i'\in [n]}\sum_{r'=1}^{k_{i'}} [la_{i,i'}]_qd^{-lm_{i,i'}} x_{i',r'}^l-\delta_{i,p}\frac{[l]_q}{l}q^lu^l\right)\cdot G\
    \mathrm{for}\ l\ne 0,$$
  $$\pi^p_{u,\overline{c}}(f_{i,k})G=
    \frac{k_ic_i^{-1}}{q^{-1}-q}\left(\underset{z=0}{\mathrm{Res}}+\underset{z=\infty}{\mathrm{Res}}\right)
    \frac{z^kG(\{x_{i',r'}\}_{\mid x_{i,k_i}\mapsto z})}{\prod_{i'}\prod_{r'=1}^{k_{i'}-\delta_{i,i'}}\omega_{i',i}(\frac{x_{i',r'}}{z})}\frac{dz}{z}.$$
Here $k\in \ZZ, \bar{c}=(c_0,\ldots,c_{n-1})\in (\CC^\times)^{[n]},\ G\in S_{1,p}(u)_{\overline{k}}$ and $|\overline{k}|:=\sum_{i\in [n]} k_i$.
\end{prop}

\begin{rem}\label{action of cartan current}
 Formulas of Proposition~\ref{extension to toroidal action} can be equivalently written in the following form
\begin{equation}\label{e-action}
  \pi^p_{u,\overline{c}}(q^{d_2})G=q^{-\frac{p(n-p)}{2}+|\overline{k}|}\cdot G,\
  \pi^p_{u,\overline{c}}(e_i(z))G=\delta\left(\frac{x_i}{z}\right)\star G,
\end{equation}
\begin{equation}\label{psi-action}
  \pi^p_{u,\overline{c}}(\psi^{\pm}_i(z))G=
  \left(\prod_{r=1}^{k_i} \frac{q^2z-x_{i,r}}{z-q^2x_{i,r}}\cdot \prod_{r=1}^{k_{i+1}}\frac{z-qdx_{i+1,r}}{qz-dx_{i+1,r}}
  \cdot \prod_{r=1}^{k_{i-1}}\frac{dz-qx_{i-1,r}}{qdz-x_{i-1,r}}\cdot \phi(z/u)^{\delta_{i,p}}\right)^\pm\cdot G,
\end{equation}
\begin{equation}\label{f-action}
  \pi^p_{u,\overline{c}}(f_i(z))G=\frac{k_ic_i^{-1}}{q^{-1}-q}\cdot
  \left\{\left(\frac{G(\{x_{i',r'}\}_{\mid x_{i,k_i}\mapsto z})}{\prod_{i'}\prod_{r'=1}^{k_{i'}-\delta_{i,i'}}\omega_{i',i}(\frac{x_{i',r'}}{z})}\right)^+-
  \left(\frac{G(\{x_{i',r'}\}_{\mid x_{i,k_i}\mapsto z})}{\prod_{i'}\prod_{r'=1}^{k_{i'}-\delta_{i,i'}}\omega_{i',i}(\frac{x_{i',r'}}{z})}\right)^-\right\},
\end{equation}
where $g(z)^\pm$ denotes the expansion of a rational function $g(z)$ in $z^{\mp 1}$, respectively.
\end{rem}

\begin{proof}[Proof of Proposition~\ref{extension to toroidal action}]
$\ $

 We need to check the compatibility of the given assignment $\pi^p_{u,\overline{c}}$
with the defining relations~(\ref{T0.1}--\ref{T7.2}). The only nontrivial of those are
(\ref{T3},\ref{T4},\ref{T6},\ref{T7.2}). To check~(\ref{T3},\ref{T6}), we use formulas~(\ref{psi-action},\ref{f-action})
together with an obvious identity $\frac{\omega_{i,j}(z/w)}{\omega_{j,i}(w/z)}=g_{a_{i,j}}(d^{m_{i,j}}z/w)$ for any $i,j\in [n]$.
The verification of~(\ref{T7.2}) boils down to the identity
  $$\underset{z_1,z_2}\Sym
    \left(\frac{\omega_{i,i}(z_1/z_2)^{-1}}{\omega_{i,i\pm 1}(z_1/w)\omega_{i,i\pm 1}(z_2/w)}-\frac{(q+q^{-1})\omega_{i,i}(z_1/z_2)^{-1}}{\omega_{i,i\pm 1}(z_1/w)\omega_{i\pm 1,i}(w/z_2)}+
    \frac{\omega_{i,i}(z_1/z_2)^{-1}}{\omega_{i\pm 1,i}(w/z_1)\omega_{i\pm 1,i}(w/z_2)}\right)=0.$$
Finally, the verification of (T4) is based on the observation that the $k_i+1-\delta_{i,j}$ different summands
from the symmetrization appearing in $\pi^p_{u,\overline{c}}(e_i(z))\pi^p_{u,\overline{c}}(f_j(w))G$
cancel the $k_i+1-\delta_{i,j}$ terms (out of $k_i+1$) from the symmetrization appearing in
$\pi^p_{u,\overline{c}}(f_j(w))\pi^p_{u,\overline{c}}(e_i(z))G$.
\end{proof}

\noindent
 $\bullet$ \emph{Shuffle modules $S(\underline{u})$.}

The above construction admits a ``higher rank'' generalization. For any $\overline{r}\in \NN^{[n]}$, consider
  $$\underline{u}=(u_{0,1},\ldots, u_{0,l_0};\ldots; u_{n-1,1}, \ldots, u_{n-1,l_{n-1}}) \ \mathrm{with}\ u_{i,s}\in \CC^\times.$$
Define $S(\underline{u})=\oplus_{\overline{k}\in \NN^{[n]}} S(\underline{u})_{\overline{k}}$ completely analogously to
$S_{1,p}(u)$ with the following modifications:

\noindent
(i$'$) \emph{Pole conditions} for a degree $\overline{k}$ function $F$ should read as follows:
  $$F=\frac{f(x_{0,1},\ldots,x_{n-1,k_{n-1}})}{\prod_{i\in [n]}\prod_{r\leq k_i}^{r'\leq k_{i+1}}(x_{i,r}-x_{i+1,r'})\cdot
    \prod_{i\in [n]}\prod_{s=1}^{l_i}\prod_{r=1}^{k_i} (x_{i,r}-u_{i,s})},\
    f\in (\CC[x_{i,r}^{\pm 1}]_{i\in [n]}^{1\leq r\leq k_i})^{\prod \mathfrak{S}_{k_i}}.$$

\noindent
(iii$'$) \emph{Second kind wheel conditions} for such $F$ should read as follows:
  $$f(\{x_{i,r}\})=0\ \mathrm{once}\ x_{i,r_1}=u_{i,s}\ \mathrm{and}\ x_{i,r_2}=q^2u_{i,s}\ \mathrm{for\ some}\
    i\in [n], 1\leq s\leq l_i, 1\leq r_1,r_2\leq k_i,$$
where $f(\{x_{i,r}\}):=\prod_{i\in [n]}\prod_{s=1}^{l_i}\prod_{r=1}^{k_i} (x_{i,r}-u_{i,s})\cdot F(\{x_{i,r}\})$.

\noindent
Let us endow $S(\underline{u})$ with an $S$-bimodule structure by applying formulas~(\ref{left}) and~(\ref{right}) with
  $$\prod_{r=1}^{k_p} \phi(x_{p,r}/u)\rightsquigarrow \prod_{i\in [n]}\prod_{s=1}^{l_i}\prod_{r=1}^{k_i}\phi(x_{i,r}/u_{i,s}).$$
The resulting left $^{'}\ddot{U}^{+}$-action on $S(\underline{u})$ can be extended to the $^{'}\ddot{U}_{q,d}(\ssl_n)$-action,
denoted $\pi_{\underline{u},\overline{c}}$.
The latter is defined by the formulas~(\ref{e-action}--\ref{f-action}) with the following two modifications:
  $$\phi(z/u)^{\delta_{i,p}}\rightsquigarrow \prod_{s=1}^{l_i} \phi(z/u_{i,s}),\
    q^{-\frac{p(n-p)}{2}}\rightsquigarrow q^{-\sum_{p=0}^{n-1} l_p\cdot\frac{p(n-p)}{2}}.$$
Let $\textbf{1}_{\underline{u}}$ denote the element $1\in S(\underline{u})_{(0,\ldots,0)}$.
The following is obvious:

\begin{lem}\label{faithful}
For $X\in\ ^{'}\ddot{U}^+\cdot {^{'}\ddot{U}^0}$, we have
$\pi_{\underline{u},\overline{c}}(X)\textbf{1}_{\underline{u}}=0$ for all $\underline{u},\overline{c}$
if and only if $X=0$.
\end{lem}

\noindent
$\bullet$ \emph{Shuffle modules $S_{1,p}^K(u)$ and $S^{\underline{K}}(\underline{u})$.}

 Another generalization of $S_{1,p}(u)$ is provided by the $S$-bimodules $S_{1,p}^K(u)$.
As a vector space, $S_{1,p}^K(u)$ is defined similarly to
$S_{1,p}(u)$ but without imposing the second kind wheel conditions.
The $S$-bimodule structure on $S_{1,p}^K(u)$ is defined by the formulas~(\ref{left}) and~(\ref{right}) with the only change
  $$\phi(t)\rightsquigarrow\phi^K(t):=\frac{K^{-1}\cdot t-K}{t-1}.$$
The resulting left $^{'}\ddot{U}^+$-action can be extended to the
$^{'}\ddot{U}_{q,d}(\ssl_n)$-action on $S_{1,p}^K(u)$, denoted $\pi^{p,K}_{u,\overline{c}}$,
defined by the formulas~(\ref{e-action}--\ref{f-action}) with the only change $\phi\rightsquigarrow \phi^K$.

 It is clear how to define the ``higher rank'' generalization $S^{\underline{K}}(\underline{u})$,
equip it with an $S$-bimodule structure, and extend the resulting left $^{'}\ddot{U}^+$-action to the
$^{'}\ddot{U}_{q,d}(\ssl_n)$-action $\pi^{\underline{K}}_{\underline{u},\overline{c}}$ on $S^{\underline{K}}(\underline{u})$.


\section{Identification of representations}

 In this section, we establish relations between representations
$\tau^p_{u,\bar{c}},\ \pi^p_{u,\bar{c}},\ \rho^p_{u,\bar{c}}$.
As before, we assume $q_1,q_2,q_3$ are generic in the sense of~(\ref{generic}).


\subsection{Isomorphism $\bar{\pi}^p_{u,\bar{c}}\simeq \tau^p_{u,\bar{c}}$}
 $\ $

 Fix $0\leq p\leq n-1, u\in \CC^\times, \bar{c}\in (\CC^\times)^{[n]}$.
Recall the action $\pi^p_{u,\bar{c}}$ of $^{'}\ddot{U}_{q,d}(\ssl_n)$ on $S_{1,p}(u)$
from Proposition~\ref{extension to toroidal action}. Define
  $$S':=\underset{\overline{k}\ne (0,\ldots,0)}\bigoplus S_{\overline{k}}\subset S.$$
Consider a $\CC$-vector subspace
  $$V_0:=S_{1,p}(u)\star S'=\mathrm{span}_\CC\{G\star F |\ G\in S_{1,p}(u), F\in S'\}\subset S_{1,p}(u).$$
The following result is straightforward and its proof is left to the interested reader:
\begin{lem}\label{preserves}
  The subspace $V_0$ of $S_{1,p}(u)$ is invariant under the action $\pi^p_{u,\bar{c}}$ of $^{'}\ddot{U}_{q,d}(\ssl_n)$.
\end{lem}


Let $\bar{\pi}^p_{u,\bar{c}}$ denote the corresponding $^{'}\ddot{U}_{q,d}(\ssl_n)$-action
on the factor space $\bar{S}_{1,p}(u):=S_{1,p}(u)/V_0.$

\begin{thm}\label{main1}
  We have an isomorphism of $^{'}\ddot{U}_{q,d}(\ssl_n)$-modules $\bar{\pi}^p_{u,\bar{c}}\simeq \tau^p_{u,\bar{c}}$.
\end{thm}

\begin{cor}
  If $q_1,q_2,q_3$ are generic in the sense of~(\ref{generic}), then $\bar{\pi}^p_{u,\bar{c}}$ is irreducible.
\end{cor}

\begin{proof}[Proof of Theorem~\ref{main1}]
$\ $

 By Proposition~\ref{Fock module}, $\tau^p_{u,\bar{c}}$ is an irreducible
$^{'}\ddot{U}_{q,d}(\ssl_n)$-representation generated by $|\emptyset\rangle$.
Moreover, both $\bar{\bf{1}}_u\in \bar{S}_{1,p}(u)$ (the image of ${\bf{1}}_u$) and $|\emptyset\rangle \in F^{(p)}(u)$ are
the lowest weight vectors of the same weight.
Therefore, it suffices to estimate dimensions of the graded components of $\bar{S}_{1,p}(u)$:
\begin{equation}\tag{$\heartsuit$}
  \sum_{\overline{k}\in \NN^{[n]}}^{|\overline{k}|=m}\dim \bar{S}_{1,p}(u)_{\overline{k}}=p(m)\ \ \forall\ m\in \NN,
\end{equation}
where $p(m)$ stays for the number of size $m$ partitions.

\medskip
\noindent
 \emph{Descending filtration.}

 To prove ($\heartsuit$), we equip $S^m_{1,p}(u):=\underset{|\overline{k}|=m} \oplus S_{1,p}(u)_{\overline{k}}$
with a filtration $\{S^{m,\lambda}_{1,p}(u)\}_\lambda$ labeled by all size $\leq m$ partitions $\lambda$.
We define $S^{m,\lambda}_{1,p}(u)$ via the specialization maps $\rho_\lambda$ introduced below as
  $$S^{m,\lambda}_{1,p}(u):=\bigcap_{\mu\succ \lambda} \mathrm{Ker}(\rho_\mu)\subset S^m_{1,p}(u),$$
where $\succ$ denotes the lexicographic order on the set of size $\leq m$ partitions.

 Consider the $[n]$-coloring of the Young diagram $\lambda$
by assigning $c(\square):=p-a+b\ (\mathrm{mod}\ n)\in [n]$ to a box
$\square=(a,b)\in \lambda$ located at the $b$-th row and $a$-th column ($1\leq b\leq \lambda'_1, 1\leq a\leq\lambda_b$).
Define
  $$\overline{k}^{\lambda}:=(k^\lambda_0,\ldots,k^\lambda_{n-1})\in \NN^{[n]},\
    \mathrm{where}\ k^\lambda_i=\# \{\square\in \lambda\mid c(\square)=i\}.$$

\begin{rem}
 We denote $\tau^p_{u,(1,\ldots,1)}$ simply by $\tau^p_u$.
Note that the map $|\lambda\rangle\mapsto \prod_{\square\in \lambda} c_{c(\square)}\cdot  |\lambda\rangle$
induces an isomorphism of $^{'}\ddot{U}_{q,d}(\ssl_n)$-representations
$\tau^p_u\iso \tau^p_{u,\bar{c}}$ for any $\bar{c}\in (\CC^\times)^{[n]}$.
\end{rem}

 Let us fill the boxes of $\lambda$ by entering $q_1^{a-1}q_3^{b-1}u$ into the box $(a,b)\in \lambda$.
For $F\in S_{1,p}(u)_{\overline{k}}$, we would like to specialize $\overline{k}^\lambda$ variables
to the corresponding entries of $\lambda$. Such a naive substitution produces zeroes
in numerators and denominators, so we need to modify it properly to get $\rho_\lambda$.

\medskip
\noindent
 \emph{Specialization maps $\rho_\lambda$.}

 For $F\in S_{1,p}(u)_{\overline{k}}$, we set $\rho_\lambda(F)=0$ if $\overline{k}-\overline{k}^\lambda\notin \NN^{[n]}$.
If $\overline{l}:=\overline{k}-\overline{k}^\lambda\in \NN^{[n]}$, we do the following:

$\circ$
 First, we consider the corner box $\square=(1,1)\in \lambda$ of color $p$ and specialize $x_{p,k_p}\mapsto u$.
Since $F$ has the first order pole at $x_{p,k_p}=u$, the following is well-defined:
  $$\rho^{(1)}_\lambda(F):=[(x_{p,k_p}-u)\cdot F]_{\mid x_{p,k_p}\mapsto u}.$$

$\circ$
 Next, we specialize more variables to the entries of the remaining boxes from the first row and the first column.
For every box $(a+1,1)\in \lambda\ (0<a< \lambda_1)$ of color $p-a$,
we choose an unspecified yet variable of the
$(p-a)$-th family $\{x_{p-a,\cdot}\}$ and set it to $q_1^au$.
Likewise, for every box $(1,b+1)\in \lambda\ (0<b< \lambda'_1)$,
we choose an unspecified yet variable of the
$(p+b)$-th family $\{x_{p+b,\cdot}\}$ and set it to $q_3^bu$.
 We perform this procedure step-by-step moving from $(1,1)$ to the right and then from $(1,1)$ up.
We denote the resulting specialization of $F$ by $\rho^{(\lambda_1+\lambda'_1-1)}_\lambda(F)$.

$\circ$
 If $(2,2)\notin \lambda$, we set $\rho_\lambda(F):=\rho^{(\lambda_1+\lambda'_1-1)}_\lambda(F)$.
If $\lambda$ contains $(2,2)$, we
would like to specify another variable of the $p$-th family, say $x_{p,k_p-1}$, to
$q_1q_3u$. Due to the first kind wheel conditions,
the function $\rho^{(\lambda_1+\lambda'_1-1)}_\lambda(F)$ has
zero at $x_{p,k_p-1}=q_1q_3u$. Hence, the following is well-defined:
  $$\rho^{(\lambda_1+\lambda'_1)}_\lambda(F):=
    \left[\frac{1}{x_{p,k_p-1}-q_1q_3u}\cdot \rho^{(\lambda_1+\lambda'_1-1)}_\lambda(F)\right]_{\mid x_{p,k_p-1}\mapsto q_1q_3u}.$$

$\circ$
 Next, we start moving from $(2,2)$ to the right and then from $(2,2)$ up.
On each step, we specialize
the corresponding $x_{\cdot,\cdot}$-variable to the prescribed entry of the
diagram. However, due to the first kind wheel conditions, we
have to eliminate order $1$ zeros as above.

$\circ$
 Performing this procedure $|\lambda|$ times, we finally obtain
$\rho^{(|\lambda|)}_\lambda(F)\in \CC(\{x_{i,r}\}_{i\in [n]}^{1\leq r\leq l_i})$.
Set
  $$\rho_\lambda(F):=\rho^{(|\lambda|)}_\lambda(F).$$

\noindent
\emph{Key properties of $\rho_\lambda$.}

Tracing back the contribution of the first and second kind wheel conditions, we find that
\begin{equation*}
  \rho_\lambda\colon S_{1,p}(u)_{\overline{k}^\lambda+\overline{l}}\longrightarrow S_{\overline{l}}\cdot G_{\overline{l},\lambda}:=
  \{F'\cdot G_{\overline{l},\lambda}|F'\in S_{\overline{l}}\},
\end{equation*}
where
\begin{multline*}
  G_{\overline{l},\lambda}=
  \prod_{r=1}^{l_p}\frac{x_{p,r}-q^2u}{x_{p,r}-u}\times\\
  \frac{\prod_{\square=(a,b)\in X^+_\lambda} \prod_{r=1}^{l_{c(\square)}}(x_{c(\square),r}-q_1^{a-1}q_3^{b-1}u)\cdot
        \prod_{\square=(a,b)\in X^-_\lambda} \prod_{r=1}^{l_{c(\square)}}(x_{c(\square),r}-q_1^{a-1}q_3^{b-1}u)}
  {\prod_{\square=(a,b)\in \lambda}
  \left\{\prod_{r=1}^{l_{c(\square)-1}}(x_{c(\square)-1,r}-q_1^{a-1}q_3^{b-1}u) \prod_{r=1}^{l_{c(\square)+1}}(x_{c(\square)+1,r}-q_1^{a-1}q_3^{b-1}u)\right\}}.
\end{multline*}
Here the set $X^+_\lambda\subset \ZZ^2$ consists of those $(a,b)\in \ZZ^2$ such that
$(a+1,b)\& (a+1,b+1)\in \lambda\ \mathrm{or}\ (a,b+1)\& (a+1,b+1)\in \lambda$, while
$X^-_\lambda\subset \ZZ^2$ consists of those $(a,b)\in \ZZ^2$ such that
$(a-1,b)\& (a-1,b-1)\in \lambda\ \mathrm{or}\ (a,b-1)\& (a-1,b-1)\in \lambda.$

 For $F\in S^{|\lambda|+|\overline{l}|,\lambda}_{1,p}(u)_{\overline{k}^\lambda+\overline{l}}$,
we further have  $\rho_\lambda(F)\in S_{\overline{l}}\cdot G_{\overline{l},\lambda}Q_{\overline{l},\lambda}$, where
  $$Q_{\overline{l},\lambda}=
    \prod_{r=1}^{l_{p-\lambda_1}} (x_{p-\lambda_1,r}-q_1^{\lambda_1}u)\cdot
    \prod_{b\geq 1}^{\lambda_{b+1}<\lambda_b} \prod_{r=1}^{l_{p-\lambda_{b+1}+b}} (x_{p-\lambda_{b+1}+b,r}-q_1^{\lambda_{b+1}}q_3^bu).$$
Our next result establishes two crucial properties of $\rho_\lambda$.

\begin{lem}\label{properties of evaluation}
(a) If $\overline{k}-\overline{k}^\lambda\notin \NN^{[n]}$, then
$\rho_\lambda(S_{1,p}(u)_{\overline{k}}\star S_{\overline{l}})=0$
for any $\overline{l}\in \NN^{[n]}$.

\noindent
(b)  We have
      $\rho_\lambda(S^{|\lambda|+|\overline{l}|,\lambda}_{1,p}(u)_{\overline{k}^\lambda+\overline{l}})=
       \rho_\lambda(S_{1,p}(u)_{\overline{k}^\lambda}\star  S_{\overline{l}})$
for any $\overline{l}\in \NN^{[n]}$.
\end{lem}

\begin{proof}[Proof of Lemma~\ref{properties of evaluation}]
$\ $

(a) For $F_1\in S_{1,p}(u)_{\overline{k}}$ and $F_2\in S_{\overline{l}}$,
let us evaluate the $\rho_\lambda$-specialization of any summand from $F_1\star F_2$.
In what follows, we say that $q_1^aq^b_3u$ gets into $F_2$
in the chosen summand if the $x_{\cdot,\cdot}$-variable which is specialized to
$q_1^aq^b_3u$ enters $F_2$ rather than $F_1$.
If $u$ gets into $F_2$, we automatically get zero once we apply  $\rho^{(1)}_\lambda$.
A simple inductive argument shows that if at least one of the
variables $\{q_1^au\}_{a=1}^{\lambda_1-1} \cup \{q_3^bu\}_{b=1}^{\lambda'_1-1}$ gets into $F_2$,
we also obtain zero after applying $\rho^{(a+1)}_\lambda$ or $\rho^{(\lambda_1+b)}_\lambda$ since the corresponding
$\omega_{\cdot,\cdot}$-factor is zero.
 If $q_1q_3u$ gets into $F_2$, but all the entries from the first
hook of $\lambda$ get into $F_1$, then there are two zero
$\omega_{\cdot,\cdot}$-factors, and so
we get zero after applying $\rho^{(\lambda_1+\lambda'_1)}_\lambda$, etc.
 However, not all the specialized variables get into $F_1$ as $\overline{k}-\overline{k}^\lambda\notin \NN^{[n]}$.
Hence, the $\rho_\lambda$-specialization of this summand is zero, and so $\rho_\lambda(F_1\star F_2)=0$.

(b) For $F_1\in S_{1,p}(u)_{\overline{k}^\lambda},F_2\in S_{\overline{l}}$,
the specialization $\rho_\lambda(F_1\star F_2)$ is a sum of $\rho_\lambda$-specializations applied to each
summand from $F_1\star F_2$.
According to (a), only one such specialization is nonzero and we have
 $\rho_\lambda(F_1\star F_2)=\prod_{i\in [n]}\frac{k^{|\lambda|}_i!l_i!}{(k^{|\lambda|}_i+l_i)!}
  \rho_\lambda(F_1)\cdot F_2(\{x_{i,r}\}_{i\in [n]}^{1\leq r\leq l_i})\cdot P,$
where $P$ denotes the product of the corresponding
$\omega_{\cdot,\cdot}$-factors:
 $P=\prod_{\square=(a,b)\in \lambda} \prod_{i\in [n]} \prod_{r=1}^{l_i} \omega_{c(\square),i}(q_1^{a-1}q_3^{b-1}u/x_{i,r}).$
It is straightforward to check that
 $P=\nu\cdot G_{\overline{l},\lambda}Q_{\overline{l},\lambda}$ with $\nu\in \CC^\times$.
To complete the proof of (b), it remains to provide $F_1\in S_{1,p}(u)_{\overline{k}^\lambda}$
such that $\rho_\lambda(F_1)\ne 0$.
To achieve this, we set
 $F_1=K_{\bar{\lambda}'_{\lambda_1}}\star \cdots \star K_{\bar{\lambda}'_1}\cdot \prod_{r=1}^{\overline{k}^\lambda_p} (x_{p,r}-u)^{-1}$,
where $\bar{\lambda}'_r\in \NN^{[n]}$ is prescribed by the coloring of the $r$-th column of $\lambda$ and
 $K_{\overline{m}}:=
  \prod_{i\in [n]}\prod_{1\leq r\ne r'\leq m_i} (x_{i,r}-q^{-2}x_{i,r'})\cdot
  \prod_{i\in [n]} \prod_{1\leq r\leq m_i}^{1\leq r'\leq m_{i+1}} \frac{x_{i,r}-q_1x_{i+1,r'}}{x_{i,r}-x_{i+1,r'}}$.
\end{proof}

\noindent
\emph{Proof of ($\heartsuit$).}

Now we are ready to deduce~($\heartsuit$), completing our proof of Theorem~\ref{main1}.
Note that
  $$\dim S^m_{1,p}(u)=\sum_{\lambda: |\lambda|\leq m} \dim \mathrm{gr}_\lambda (S^m_{1,p}(u)),\ \
    \dim \bar{S}^m_{1,p}(u)=\sum_{\lambda: |\lambda|\leq m} \dim \mathrm{gr}_\lambda (\bar{S}^m_{1,p}(u)),$$
where the filtration $\{\bar{S}^{m,\lambda}_{1,p}(u)\}_\lambda$ on $\bar{S}^m_{1,p}(u)$ is
induced by the filtration $\{S^{m,\lambda}_{1,p}(u)\}_\lambda$ on $S^m_{1,p}(u)$.
The $\rho_\lambda$-specialization identifies $\mathrm{gr}_\lambda (S^m_{1,p}(u))$ with $\rho_\lambda (S^{m,\lambda}_{1,p}(u))$.
This observation and Lemma~\ref{properties of evaluation} imply that $\mathrm{gr}_\lambda (\bar{S}^m_{1,p}(u))$ is zero
if $|\lambda|<m$ and is $1$-dimensional if $|\lambda|=m$. This proves~($\heartsuit$).
\end{proof}


\subsection{Generalizations to $S(\underline{u})$ and $S^{\underline{K}}(\underline{u})$}
 $\ $

The result of Theorem~\ref{main1} can be generalized in both directions mentioned in Section 1.8.
Recall the $^{'}\ddot{U}_{q,d}(\ssl_n)$-action $\pi_{\underline{u},\overline{c}}$ on the space $S(\underline{u})$,
which preserves the subspace $S(\underline{u})\star S'$ (see Lemma~\ref{preserves}).
Let $\bar{\pi}_{\underline{u},\bar{c}}$ denote the induced $^{'}\ddot{U}_{q,d}(\ssl_n)$-action on
$\bar{S}(\underline{u}):=S(\underline{u})/(S(\underline{u})\star S')$.
We call $\underline{u}=\{u_{i,s}\}_{i\in [n]}^{1\leq s\leq l_i}$ \emph{generic}
if $\{(i,u_{i,s},(1,\ldots,1))\}$ is \emph{generic} in the sense of subsection~\ref{section Fock}.

\begin{thm}\label{main2}
For a generic $\underline{u}=\{u_{i,s}\}_{i\in [n]}^{1\leq s\leq l_i}$, we have an isomorphism of $^{'}\ddot{U}_{q,d}(\ssl_n)$-modules
  $$\bar{\pi}_{\underline{u},\bar{c}}\simeq \otimes_{i=0}^{n-1}\otimes_{s=1}^{l_i} \tau^i_{u_{i,s}}.$$
\end{thm}

\begin{proof}[Proof of Theorem~\ref{main2}]
$\ $

 The proof of this theorem goes along the same lines as for the case $\sum l_i=1$ from above.
According to Lemma~\ref{tensor}, $\otimes_{i=0}^{n-1}\otimes_{s=1}^{l_i} \tau^i_{u_{i,s}}$
is a well-defined, irreducible, lowest weight representation
generated by the lowest weight vector
$|\emptyset\rangle_{\underline{u}}:=\otimes_{i=0}^{n-1}\otimes_{s=1}^{l_i} |\emptyset\rangle$.
On the other hand, the vector $\bar{\bf{1}}_{\underline{u}}\in \bar{S}(\underline{u})$
is the lowest weight vector of the same weight as
$|\emptyset\rangle_{\underline{u}}$.
 Therefore, it suffices to compare the dimensions.
This can be accomplished as above by using the specialization maps
$\rho_{\underline{\lambda}}$ with $\underline{\lambda}=\{\lambda^{(0,1)},\ldots,\lambda^{(n-1,l_{n-1})}\}$
(they are defined similarly to $\rho_\lambda$, but the entry of $\square=(a,b)\in \lambda^{(i,s)}$
is set to be $q_1^{a-1}q_3^{b-1}u_{i,s}$, while its color is $c(\square):=i-a+b\ (\mathrm{mod}\ n)\in [n]$).
\end{proof}

Another generalization of Theorem~\ref{main1} establishes an isomorphism of
$\bar{\pi}^{\underline{K}}_{\underline{u},\bar{c}}$ and tensor products of Macmahon modules
for generic parameters.
For simplicity of our exposition, we restrict attention to the case of
$\pi^{p,K}_{u,c}$ for \emph{generic} $K$ (that is, $K\notin q^\ZZ d^\ZZ$).
Let $\bar{\pi}^{p,K}_{u,c}$ denote the induced $^{'}\ddot{U}_{q,d}(\ssl_n)$-action
on the factor space $\bar{S}^{K}_{1,p}(u):=S^{K}_{1,p}(u)/(S^{K}_{1,p}(u)\star S').$

\begin{thm}\label{main3}
  We have an isomorphism of $^{'}\ddot{U}_{q,d}(\ssl_n)$-modules $\bar{S}^{K}_{1,p}(u)\simeq M^{(p)}(u,K)$.
\end{thm}

\begin{proof}[Proof of Theorem~\ref{main3}]
$\ $

We apply same proof as for Theorem~\ref{main1}, while now the filtration is
parametrized by plane partitions $\bar{\lambda}=(\lambda^{(1)},\lambda^{(2)},\ldots)$.
Here, we fill the boxes of $\bar{\lambda}$ by entering $q_1^{a-1}q_2^{c-1}q_3^{b-1}u$
into the box $\square=(a,b)\in \lambda^{(c)}$ and define the specialization maps
$\rho_{\bar{\lambda}}$ as before. Whence the arguments from our proof of Theorem~\ref{main1} apply word by word.
Note that the only place where we used the second kind wheel conditions
was the appearance of the factor $\prod_{r=1}^{l_p} (x_{p,r}-q^2u)$ in $G_{\overline{l},\lambda}$.
This is now compensated by a change of $Q_{\overline{l},\lambda}$--the factor which keeps track of the filtration depth.
\end{proof}


\subsection{Isomorphism $\rho^{p,\varpi}_{v,\bar{c}}\simeq \tau^{*,p}_{u}$}
$\ $

 Given a representation $\rho$ of an algebra $B$ on a vector space
$V$ and an algebra homomorphism $\sigma\colon A\to B$, we use $\rho^\sigma$
to denote the corresponding representation of $A$ on $V$: $\rho^\sigma(x)=\rho(\sigma(x))$.
To simplify our notation, we define $^{'}\ddot{U}_{q,d}(\ssl_n)$-modules $\rho^{p,\varpi}_{v,\bar{c}}:=(\rho^p_{v,\bar{c}})^\varpi$
and $\tau^{*,p}_{u}:={^*(\tau^p_u)}$.
Actually, the left dual and the right dual modules of $\tau^p_u$ are isomorphic: $^*(\tau^p_u)\simeq (\tau^p_u)^*$.

\begin{thm}\label{main4}
 For any $0\leq p\leq n-1, v\in \CC^\times, \bar{c}\in (\CC^\times)^{[n]}$, we have
an isomorphism of $^{'}\ddot{U}_{q,d}(\ssl_n)$-modules $\rho^{p,\varpi}_{v,\bar{c}}\simeq \tau^{*,p}_{u}$,
where $u:=(-1)^{\frac{(n-2)(n-3)}{2}}q^{-1}d^{-p-(n-1)\delta_{p,0}}\cdot (c_0\cdots c_{n-1})^{-1}.$
\end{thm}

\begin{cor}\label{irrelevence2}
 For any $0\leq p\leq n-1,\ v,v'\in \CC^\times,\ \bar{c},\bar{c}'\in (\CC^\times)^{[n]}$
with $\prod_{i\in [n]} c_i=\prod_{i\in [n]} c_i'$, we have an
isomorphism of  $\ddot{U}^{'}_{q,d}(\ssl_n)$-representations
$\rho^p_{v,\bar{c}}\simeq \rho^p_{v',\bar{c}'}$.
\end{cor}


\begin{proof}[Proof of Theorem~\ref{main4}]
$\ $

 Our proof consists of three steps. First, we verify that
both $v_0\otimes e^{\bar{\Lambda}_p}$ and $|\emptyset\rangle^*$
have the same eigenvalues with respect to the \emph{finite Cartan subalgebra}
$\CC[\psi^{\pm 1}_{0,0},\ldots,\psi^{\pm 1}_{n-1,0}, q^{\pm d_2}]$.
 Second, we show that both vectors are annihilated by
$e_{i,k}$-action for any $i\in [n], k\in \ZZ$.
Finally, we prove that they have the same eigenvalues with respect
to $\psi_{i,l}$-action for any $i\in [n], l\in \ZZ\backslash \{0\}$.

\medskip
\noindent
\emph{Step 1: Comparing weights with respect to $\CC[\psi^{\pm 1}_{0,0},\ldots,\psi^{\pm 1}_{n-1,0}, q^{\pm d_2}]$}.

 According to Proposition~\ref{Saito representation}, elements $\psi_{i,0}, \gamma, q^{d_1}$
act on $v_0\otimes e^{\bar{\Lambda}_p}$ via multiplication by $q^{\langle\bar{h}_i,\bar{\Lambda}_p\rangle}, q, 1$, respectively.
Combining this with Proposition~\ref{explicit formulas for Miki}(a), we get
  $$\rho^{p,\varpi}_{v,\bar{c}}(q^{d_2})v_0\otimes e^{\bar{\Lambda}_p}=q^{\frac{p(n-p)}{2}}\cdot v_0\otimes e^{\bar{\Lambda}_p},\
    \rho^{p,\varpi}_{v,\bar{c}}(\psi_{i,0}) v_0\otimes e^{\bar{\Lambda}_p}=q^{\delta_{i,p}}\cdot v_0\otimes e^{\bar{\Lambda}_p}\
    \ \forall\ i\in [n].$$
We also have
 $\tau^p_u(q^{d_2})|\emptyset\rangle=q^{-\frac{p(n-p)}{2}}\cdot |\emptyset\rangle$
and
 $\tau^p_u(\psi_{i,0})|\emptyset\rangle=q^{-\delta_{i,p}}\cdot |\emptyset\rangle.$
Therefore, the vectors $v_0\otimes e^{\bar{\Lambda}_p}\in  \rho^{p,\varpi}_{v,\bar{c}}$ and
$|\emptyset\rangle^* \in \tau^{*,p}_u$ have the same weights with respect to
$\CC[\psi^{\pm 1}_{0,0},\ldots,\psi^{\pm 1}_{n-1,0}, q^{\pm d_2}]$.

\begin{rem}
 This explains the appearance of $q^{-\frac{p(n-p)}{2}}$ in the formulas for
$\tau^p_{u,\bar{c}}(q^{d_2}),\ \pi^p_{u,\bar{c}}(q^{d_2})$.
\end{rem}

\medskip
\noindent
\emph{Step 2: Verifying an annihilation property with respect to ${^{'}\ddot{U}^{+}}$}.

 First, we prove $\rho^{p,\varpi}_{v,\bar{c}}(e_{i,0}) v_0\otimes e^{\bar{\Lambda}_p}=0$ for $i\in[n]$.
For $i\ne 0$, this is clear as $\langle\bar{h}_i,\bar{\Lambda}_p\rangle+1>0$, while $H_{i',k}v_0=0$ for all $i'\in[n], k>0$.
For $i=0$, $\varpi(e_{0,0})=d\gamma \psi_{0,0}[\cdots[f_{1,1},f_{2,0}]_q,\cdots,f_{n-1,0}]_q$ by Proposition~\ref{explicit formulas for Miki}(a),
and the equality $\rho^{p,\varpi}_{v,\bar{c}}(e_{0,0}) v_0\otimes e^{\bar{\Lambda}_p}=0$ follows from our next result.

\begin{lem}\label{e0-annihilation}
 $\rho^p_{v,\bar{c}}([\cdots[f_{1,1},f_{2,0}]_q,\cdots,f_{n-1,0}]_q) v_0\otimes e^{\bar{\Lambda}_p}=0$.
\end{lem}

\begin{proof}[Proof of Lemma~\ref{e0-annihilation}]
$\ $

It suffices to show that any summand $f_{i_{n-1},r_{n-1}}\cdots f_{i_1,r_1}$ of the above multicommutator
(here $\{i_1, \ldots, i_{n-1}\}=[n]^\times$ and $r_k=\delta_{i_k,1}$)
acts trivially on $v_0\otimes e^{\bar{\Lambda}_p}$.
Since $-\langle\bar{h}_i,\bar{\Lambda}_p\rangle+1>0$ for $i\ne p$, we see that
 $\rho^p_{v,\bar{c}}(f_{i_1,r_1}) v_0\otimes e^{\bar{\Lambda}_p}=0$ unless $i_1=p\ne 1$.
For $i_1=p\ne 1$, we get
 $\rho^p_{v,\bar{c}}(f_{i_1,j_1}) v_0\otimes e^{\bar{\Lambda}_p}=\pm c^{-1}_{i_1}v_0\otimes e^{\bar{\Lambda}^{(1)}_p}$
with $\bar{\Lambda}^{(1)}_p:=\bar{\Lambda}_p-\bar{\alpha}_{i_1}$. The
key property of this weight is $-\langle\bar{h}_i,\bar{\Lambda}^{(1)}_p\rangle+1\geq 0$.
In particular, $\rho^p_{v,\bar{c}}(f_{i_2,r_2}) v_0\otimes e^{\bar{\Lambda}^{(1)}_p}=0$
unless $i_2=p-1\ne 1$ or $i_2=p+1\ne 1$. In the latter two
cases, the result is $\pm c^{-1}_{i_2}v_0\otimes e^{\bar{\Lambda}^{(2)}_p}$ with
$\bar{\Lambda}^{(2)}_p:=\bar{\Lambda}^{(1)}_p-\bar{\alpha}_{i_2}$
satisfying a similar property. Continuing in the same way, we finally
get to the $k$-th place with $i_k=1$ and $r_k=1$.
As $-\langle \bar{h}_1,\bar{\Lambda}^{(k-1)}_p\rangle+1\geq 0$, we have
$\rho^p_{v,\bar{c}}(f_{i_k,r_k}\cdots f_{i_1,r_1}) v_0\otimes  e^{\bar{\Lambda}_p}=0.$
\end{proof}

This completes our proof of the equality $\rho^{p,\varpi}_{v,\bar{c}}(e_{i,0}) v_0\otimes  e^{\bar{\Lambda}_p}=0$ for any $i\in [n]$.

 According to~($\ddag$) from the next step, we have
$\rho^{p,\varpi}_{v,\bar{c}}(h_{j,\pm 1}) v_0\otimes  e^{\bar{\Lambda}_p}=0$ for $j\ne p$.
Combining this formula with the relation (T5$'$)
 $[h_{j,\pm 1},e_{i,k}]=d^{\mp m_{j,i}}\gamma^{-1/2} [a_{j,i}]_q\cdot e_{i,k\pm 1}$,
one gets
  $$\rho^{p,\varpi}_{v,\bar{c}}(e_{i,k}) v_0\otimes e^{\bar{\Lambda}_p}=0\ \ \mathrm{for\ any}\ \ i\in [n], k\in \ZZ.$$

 On the other hand, the identity $S(e_i(z))=-\psi^-_i(\gamma^{-1/2}z)^{-1}e_i(\gamma^{-1}z)$ combined with the formulas
of Proposition~\ref{Fock module} imply a similar equality
$\tau^{*,p}_u(e_{i,k}) |\emptyset\rangle^*=0\ \mathrm{for\ any}\ i\in [n], k\in \ZZ.$

\medskip
\noindent
\emph{Step 3: Comparing weights with respect to ${^{'}\ddot{U}^{0}}$}.

 Let us now prove that both $v_0\otimes e^{\bar{\Lambda}_p}\in \rho^{p,\varpi}_{v,\bar{c}}$ and $|\emptyset\rangle^*\in \tau^{*,p}_u$
are eigenvectors with respect to the generators $\{\psi_{i,l}\}_{i\in [n]}^{l\ne 0}$ and have the same eigenvalues.
By the definition of $\tau^p_u$, we have
  $$\tau^{*,p}_u(\psi^\pm_{i}(z))|\emptyset\rangle^*=\phi(z/u)^{-\delta_{i,p}} |\emptyset\rangle^*\Rightarrow
    \tau^{*,p}_u(\psi_{i,\pm r})|\emptyset\rangle^*=\pm \delta_{i,p}(q-q^{-1})(q^2u)^{\pm r} |\emptyset\rangle^*$$
for any $i\in [n], r\in \ZZ_{>0}$.
Therefore, it remains to show
\begin{equation}\tag{$\ddag$}
  \rho^{p,\varpi}_{v,\bar{c}}(\psi_{i,\pm r}) v_0\otimes e^{\bar{\Lambda}_p}=
  \pm \delta_{i,p}(q-q^{-1})(q^2u)^{\pm r}v_0\otimes e^{\bar{\Lambda}_p}\
  \mathrm{for\ any}\ i\in [n], r\in \ZZ_{>0}.
\end{equation}

Our proof of ($\ddag$) is based on the following technical result.
\begin{lem}\label{technical lemma}
 We have the following equalities:
\begin{equation}\label{zero}
  \rho^{p,\varpi}_{v,\bar{c}}(f_{i,0})v_0\otimes e^{\bar{\Lambda}_p}=
  \delta_{i,p}c_p^{-1}\cdot \lambda^{\delta_{p,0}}\cdot v_0\otimes e^{-\bar{\alpha}_p}e^{\bar{\Lambda}_p},
\end{equation}
\begin{equation}\label{one}
  \rho^{p,\varpi}_{v,\bar{c}}(h_{i,-1})v_0\otimes e^{\bar{\Lambda}_p}=
  \delta_{i,p} q^{-1}u^{-1}\cdot v_0\otimes e^{\bar{\Lambda}_p},
\end{equation}
\begin{equation}\label{two}
  \rho^{p,\varpi}_{v,\bar{c}}(h_{i,1})v_0\otimes e^{\bar{\Lambda}_p}=
  \delta_{i,p} qu\cdot v_0\otimes e^{\bar{\Lambda}_p},
\end{equation}
\begin{equation}\label{three}
  \rho^{p,\varpi}_{v,\bar{c}}(h_{p,-1})v_0\otimes e^{-\bar{\alpha}_p}e^{\bar{\Lambda}_p}=
  -q^{-3}u^{-1} \cdot v_0\otimes e^{-\bar{\alpha}_p}e^{\bar{\Lambda}_p},
\end{equation}
\begin{equation}\label{four}
  \rho^{p,\varpi}_{v,\bar{c}}(h_{p,1}) v_0\otimes e^{-\bar{\alpha}_p}e^{\bar{\Lambda}_p}=
  -q^3u \cdot v_0\otimes e^{-\bar{\alpha}_p}e^{\bar{\Lambda}_p},
\end{equation}
where
 $\C:=\prod_{j\in [n]} c_j,\ \lambda:=(-1)^{\frac{(n-2)(n-3)}{2}}v^{-1}q^{-1}d^{-1}\C,\
  u:=(-1)^{\frac{(n-2)(n-3)}{2}}q^{-1}d^{-p-(n-1)\delta_{p,0}}\C^{-1}$.
\end{lem}

\begin{proof}[Proof of Lemma~\ref{technical lemma}]
$\ $

\noindent
$\circ$ For $i\ne 0$, we have $\varpi(f_{i,0})=f_{i,0}$ and
$-\langle \bar{h}_i,\bar{\Lambda}_p\rangle+1=1-\delta_{i,p}\geq 0$, so that
  $$\rho^{p,\varpi}_{v,\bar{c}}(f_{i,0})v_0\otimes e^{\bar{\Lambda}_p}=
    \delta_{i,p}c_p^{-1}\cdot v_0\otimes e^{-\bar{\alpha}_p}e^{\bar{\Lambda}_p}.\
    \checkmark$$
For $i=0$, we apply the formula for $\varpi(f_{0,0})$ from Proposition~\ref{explicit formulas for Miki}(a) to get
  $$\rho^{p,\varpi}_{v,\bar{c}} (f_{0,0})v_0\otimes e^{\bar{\Lambda}_p}=q^{-\delta_{p,0}}d^{-1}\cdot
    \rho^p_{v,\bar{c}}([e_{n-1,0},\cdots,[e_{2,0}, e_{1,-1}]_{q^{-1}}\cdots ]_{q^{-1}})v_0\otimes e^{\bar{\Lambda}_p}.$$
As $\rho^p_{v,\bar{c}}(e_{j,0}) v_0\otimes e^{\bar{\Lambda}_p}=0$ for $j\ne 0$, we see (by rewriting the above multicommutator) that
  $$\rho^p_{v,\bar{c}}([e_{n-1,0},\cdots,[e_{2,0}, e_{1,-1}]_{q^{-1}}\cdots ]_{q^{-1}})v_0\otimes e^{\bar{\Lambda}_p}=
    \rho^p_{v,\bar{c}}(e_{n-1,0})\cdots \rho^p_{v,\bar{c}}(e_{2,0})\rho^p_{v,\bar{c}}(e_{1,-1})v_0\otimes  e^{\bar{\Lambda}_p}.$$
For $p\ne 0$, the same argument as before implies
 $\rho^p_{v,\bar{c}}(e_{p,0})\cdots \rho^p_{v,\bar{c}}(e_{2,0})\rho^p_{v,\bar{c}}(e_{1,-1})v_0\otimes  e^{\bar{\Lambda}_p}=0,$
while for $p=0$ we have
  $$\rho^p_{v,\bar{c}}(e_{n-1,0})\cdots \rho^p_{v,\bar{c}}(e_{2,0})\rho^p_{v,\bar{c}}(e_{1,-1})v_0\otimes  e^{\bar{\Lambda}_p}=
    v^{-1}(c_1\cdots c_{n-1})\cdot v_0\otimes  (e^{\bar{\alpha}_{n-1}}\cdots e^{\bar{\alpha}_1}).$$
Since $e^{\bar{\alpha}_{n-1}}\cdots  e^{\bar{\alpha}_1}=(-1)^{\frac{(n-2)(n-3)}{2}}e^{-\bar{\alpha}_0}$, we finally get
  $$\rho^{0,\varpi}_{v,\bar{c}}(f_{0,0})v_0\otimes e^0=
    (-1)^{\frac{(n-2)(n-3)}{2}}v^{-1}q^{-1}d^{-1}\C\cdot c_0^{-1}v_0\otimes e^{-\bar{\alpha}_0}.\ \checkmark$$

\medskip
In what follows below, we assume $p\ne 0$.

\noindent
$\circ$ Combining the formula for $\varpi(h_{0,-1})$ from Proposition~\ref{explicit formulas for Miki}(c) with
        $$\rho^p_{v,\bar{c}}(e_{0,1})v_0\otimes  e^{\bar{\Lambda}_p}=
          \rho^p_{v,\bar{c}}(e_{2,0})v_0\otimes e^{\bar{\Lambda}_p}=\cdots=
          \rho^p_{v,\bar{c}}(e_{n-1,0})v_0\otimes e^{\bar{\Lambda}_p}=0,$$
we get
  $$\rho^{p,\varpi}_{v,\bar{c}}(h_{0,-1})v_0\otimes  e^{\bar{\Lambda}_p}=
    (-1)^nd^{n-1}\cdot \rho^p_{v,\bar{c}}(e_{0,1})\rho^p_{v,\bar{c}}(e_{n-1,0})\cdots \rho^p_{v,\bar{c}}(e_{2,0})\rho^p_{v,\bar{c}}(e_{1,-1})
    v_0\otimes  e^{\bar{\Lambda}_p}.$$
The latter is zero, since
 $\rho^p_{v,\bar{c}}(e_{p-1,0})\cdots \rho^p_{v,\bar{c}}(e_{1,-1})v_0\otimes  e^{\bar{\Lambda}_p}=
  \pm v^{-1}c_1\cdots c_{p-1}\cdot v_0\otimes e^{\bar{\Lambda}_p+\bar{\alpha}_1+\cdots+\bar{\alpha}_{p-1}}$
and $\langle \bar{h}_p,  \bar{\Lambda}_p+\bar{\alpha}_1+\cdots+\bar{\alpha}_{p-1}\rangle+1>0$.
Thus $\rho^{p,\varpi}_{v,\bar{c}}(h_{0,-1})v_0\otimes e^{\bar{\Lambda}_p}=0$ for $p\ne 0$. $\checkmark$

For $i\ne 0$, the formula for $\varpi(h_{i,-1})$ combined with
$\rho^p_{v,\bar{c}}(e_{j,0})v_0\otimes e^{\bar{\Lambda}_p}=0\  (j\ne  0)$
implies
  $$\rho^{p,\varpi}_{v,\bar{c}}(h_{i,-1})v_0\otimes  e^{\bar{\Lambda}_p}=(-1)^{i+1}d^i
    \rho^p_{v,\bar{c}}(e_{i,0})\cdots \rho^p_{v,\bar{c}}(e_{1,0})\rho^p_{v,\bar{c}}(e_{i+1,0})\cdots
    \rho^p_{v,\bar{c}}(e_{n-1,0})\rho^p_{v,\bar{c}}(e_{0,0}) v_0\otimes  e^{\bar{\Lambda}_p}.$$
If $i<p$, then $\langle \bar{h}_p,\bar{\Lambda}_p+\bar{\alpha}_0+\sum_{j=p+1}^{n-1}\bar{\alpha}_j\rangle+1>0$ and so
  $$\rho^p_{v,\bar{c}}(e_{p,0})\cdots \rho^p_{v,\bar{c}}(e_{n-1,0})\rho^p_{v,\bar{c}}(e_{0,0})v_0\otimes  e^{\bar{\Lambda}_p}=0
    \Rightarrow \rho^{p,\varpi}_{v,\bar{c}}(h_{i,-1})v_0\otimes  e^{\bar{\Lambda}_p}=0.$$
If $i>p$, then $\langle \bar{h}_p,\bar{\Lambda}_p+\bar{\alpha}_0+\sum_{j=i+1}^{n-1}\bar{\alpha}_j+\sum_{j=1}^{p-1}\bar{\alpha}_j\rangle+1>0$ and
hence $\rho^{p,\varpi}_{v,\bar{c}}(h_{i,-1})v_0\otimes e^{\bar{\Lambda}_p}=0$.
If $i=p$, then we get
  $$\rho^{p,\varpi}_{v,\bar{c}}(h_{p,-1})v_0\otimes e^{\bar{\Lambda}_p}=
    (-1)^{p+1}d^p(c_0\cdots c_{n-1}) v_0\otimes (e^{\bar{\alpha}_p}\cdots e^{\bar{\alpha}_1}e^{\bar{\alpha}_{p+1}}\cdots
    e^{\bar{\alpha}_{n-1}}e^{\bar{\alpha}_0}e^{\bar{\Lambda}_p})=$$
  $$(-1)^{\frac{(n-2)(n-3)}{2}}d^p\C\cdot v_0\otimes e^{\bar{\Lambda}_p}
    \ \ \mathrm{as}\ \
    e^{\bar{\alpha}_p}\cdots e^{\bar{\alpha}_1}e^{\bar{\alpha}_{p+1}}\cdots e^{\bar{\alpha}_{n-1}}e^{\bar{\alpha}_0}=
    (-1)^{\frac{n(n-1)}{2}+p}\cdot e^0.\
    \checkmark$$

\noindent
$\circ$ Due to similar arguments (though this case is a bit more tedious),
we have $\rho^{p,\varpi}_{v,\bar{c}}(h_{i,1}) v_0\otimes e^{\bar{\Lambda}_p}=0$ if $i\ne p$.
For $i=p$, we get
  $$\rho^{p,\varpi}_{v,\bar{c}}(h_{p,1})v_0\otimes e^{\bar{\Lambda}_p}=
    (-1)^{n+p+1}d^{-p}(c_0^{-1}\cdots c_{n-1}^{-1}) v_0\otimes (e^{-\bar{\alpha}_0}e^{-\bar{\alpha}_{n-1}}\cdots
    e^{-\bar{\alpha}_{p+1}}e^{-\bar{\alpha}_1}\cdots e^{-\bar{\alpha}_p}e^{\bar{\Lambda}_p})=$$
  $$(-1)^{\frac{(n-2)(n-3)}{2}}d^{-p}\C^{-1}\cdot v_0\otimes e^{\bar{\Lambda}_p}
    \ \ \mathrm{as}\ \
    e^{-\bar{\alpha}_0}e^{-\bar{\alpha}_{n-1}}\cdots e^{-\bar{\alpha}_{p+1}}e^{-\bar{\alpha}_1}\cdots e^{-\bar{\alpha}_p}=
    (-1)^{\frac{n(n+1)}{2}+p}\cdot e^0.
    \ \checkmark$$

\noindent
$\circ$
Arguing as above, only one summand of the corresponding multicommutator acts nontrivially:
  $$\rho^{p,\varpi}_{v,\bar{c}}(h_{p,-1})v_0\otimes e^{-\bar{\alpha}_p}e^{\bar{\Lambda}_p}=$$
  $$(-1)^{p+1}d^p(-q^{-2})\rho^p_{v,\bar{c}}(e_{p-1,0})\cdots \rho^p_{v,\bar{c}}(e_{1,0})\rho^p_{v,\bar{c}}(e_{p+1,0})\cdots
    \rho^p_{v,\bar{c}}(e_{0,0})\rho^p_{v,\bar{c}}(e_{p,0})v_0\otimes e^{-\bar{\alpha}_p}e^{\bar{\Lambda}_p}=$$
  $$(-1)^{1+\frac{(n-2)(n-3)}{2}}d^pq^{-2}\C\cdot v_0\otimes e^{-\bar{\alpha}_p}e^{\bar{\Lambda}_p}.\ \checkmark$$

\noindent
$\circ$
Arguing as above, only one summand of the corresponding multicommutator acts nontrivially:
  $$\rho^{p,\varpi}_{v,\bar{c}}(h_{p,1})v_0\otimes e^{-\bar{\alpha}_p}e^{\bar{\Lambda}_p}=$$
  $$(-1)^{p+1}d^{-p}(-q^2)\rho^p_{v,\bar{c}}(f_{p,0})\rho^p_{v,\bar{c}}(f_{0,0})\cdots
    \rho^p_{v,\bar{c}}(f_{p+1,0})\rho^p_{v,\bar{c}}(f_{1,0})\cdots
    \rho^p_{v,\bar{c}}(f_{p-1,0}) v_0\otimes e^{-\bar{\alpha}_p}e^{\bar{\Lambda}_p}=$$
  $$(-1)^{1+\frac{(n-2)(n-3)}{2}}d^{-p}q^{2}\C^{-1}\cdot v_0\otimes e^{-\bar{\alpha}_p}e^{\bar{\Lambda}_p}.\ \checkmark$$

The proofs of~(\ref{one}--\ref{four}) for $p=0$ are analogous and are left to the interested reader.
\end{proof}

\noindent
\emph{Proof of~($\ddag$)}.

Note that
 $\rho^{p,\varpi}_{v,\bar{c}}(e_{p,0}) v_0\otimes e^{-\bar{\alpha}_p}e^{\bar{\Lambda}_p}=
  c_p \lambda^{-\delta_{p,0}} \cdot v_0\otimes  e^{\bar{\Lambda}_p}$.
Combining this with the identity $[h_{p,\pm 1}, e_{p,\pm r}]=\gamma^{-1/2}(q+q^{-1}) e_{p,\pm (r+1)}$ and
the equalities~(\ref{one}--\ref{four}) of Lemma~\ref{technical lemma}, we get
  $$\rho^{p,\varpi}_{v,\bar{c}}(e_{p,\pm r}) v_0\otimes e^{-\bar{\alpha}_p}e^{\bar{\Lambda}_p}=
    c_p\lambda^{-\delta_{p,0}}(q^2u)^{\pm r}\cdot v_0\otimes e^{\bar{\Lambda}_p}\
    \mathrm{for}\ r\in \ZZ_{>0}.$$
On the other hand, we have
  $$\rho^{p,\varpi}_{v,\bar{c}}(\psi_{i,\pm r})=
    \pm (q-q^{-1})[\rho^{p,\varpi}_{v,\bar{c}}(e_{i, \pm r}), \rho^{p,\varpi}_{v,\bar{c}}(f_{i,0})]\
    \mathrm{for}\ r\in \ZZ_{>0}.$$
Since
 $\rho^{p,\varpi}_{v,\bar{c}}(e_{i,\pm r})v_0\otimes e^{\bar{\Lambda}_p}=
  \rho^{p,\varpi}_{v,\bar{c}}(f_{i,0})v_0\otimes e^{\bar{\Lambda}_p}=0$
for $i\ne p$, we get $\rho^{p,\varpi}_{v,\bar{c}}(\psi_{i,\pm r})v_0\otimes  e^{\bar{\Lambda}_p}=0$ if $i\ne p$.
The equality ($\ddag$) follows now from
  $$\rho^{p,\varpi}_{v,\bar{c}} (\psi_{p,\pm r})v_0\otimes e^{\bar{\Lambda}_p}=
    \pm (q-q^{-1}) \rho^{p,\varpi}_{v,\bar{c}} (e_{p,\pm r}) \rho^{p,\varpi}_{v,\bar{c}}(f_{p,0})v_0\otimes e^{\bar{\Lambda}_p}=
    \pm (q-q^{-1})(q^2u)^{\pm r}v_0\otimes e^{\bar{\Lambda}_p}.$$

\medskip
 The irreducibility of $\rho^p_{v,\bar{c}}$ and $\tau^p_u$
(which is guaranteed by the assumption~(\ref{generic}), see Propositions~\ref{Fock module},~\ref{Saito representation}) implies
that both $^{'}\ddot{U}_{q,d}(\ssl_n)$-representations $\rho^{p,\varpi}_{v,\bar{c}}$ and $\tau^{*,p}_u$ are irreducible.
Moreover, they are generated by the vectors $v_0\otimes e^{\bar{\Lambda}_p}$ and $|\emptyset\rangle^*$, which
are the highest weight vectors of the same weight, due to Steps 1--3. Theorem~\ref{main4} follows.
\end{proof}


\section{Matrix elements of $L$ operators}

 In this section, we study matrix elements of $L$ operators associated to $\rho^{p}_{u,\bar{c}}$.
Let us denote $\rho^p_{1,\bar{c}}$ simply by $\rho^p_{\bar{c}}$.
It suffices to work only with $\rho^p_{\bar{c}}$ as $\rho^p_{u,\bar{c}}\simeq \rho^p_{\bar{c}}$ for any $u\in \CC^\times$,
due to Corollary~\ref{irrelevence2}.
 We provide a new realization of the $S$-bimodule $S_{1,p}(u)$
as a bimodule generated by $L^{p,\bar{c}}_{\emptyset,\emptyset}$.


\subsection{Matrix elements}
$\ $

 For any $w\in W(p)_n^*$ and $v\in W(p)_n$, we consider
  $$L^{p,\bar{c}}_{w,v}:=\langle 1\otimes w| (1\otimes \rho^{p}_{\bar{c}})(R^{'}) |1\otimes v\rangle,$$
the matrix element of the universal $R$-matrix $R^{'}$ with respect to the second component.
We will mainly work with the cases $v=|\emptyset\rangle:=v_0\otimes  e^{\bar{\Lambda}_p}\in W(p)_n$ or
$w=\langle\emptyset|$--the dual of $|\emptyset\rangle$.
In what follows, we abbreviate $|\emptyset\rangle$ and $\langle \emptyset|$ simply by $\emptyset$
when they appear as indexes of matrix elements.

\begin{lem}\label{auxilary}
For $i\in [n], r\in \ZZ_{>0}, v\in W(p)_n$, we have
  $[h_{i,-r}, L^{p,\bar{c}}_{\emptyset,v}]_{q^{-r}}=(\gamma/q)^{r/2}\cdot L^{p,\bar{c}}_{\emptyset, \rho^{p}_{\bar{c}}(h_{i,-r})v}$.
\end{lem}

\begin{proof}[Proof of Lemma~\ref{auxilary}]
$\ $

We combine $\Delta(h_{i,-r})=h_{i,-r}\otimes \gamma^{-r/2}+ \gamma^{r/2}\otimes h_{i,-r}$
with $R^{'}\Delta(h_{i,-r})=\Delta^{\mathrm{op}}(h_{i,-r})R^{'}$ and apply $1\otimes \rho^{p}_{\bar{c}}$ to the resulting equality.
Comparing the matrix elements between $\langle\emptyset|$ and $v$ (with respect to the second component) recovers the claimed identity.
\end{proof}

Our first goal is to compute explicitly $L^{p,\bar{c}}_{\emptyset,\emptyset}$.
The shuffle-type formula for $L^{p,\bar{c}}_{\emptyset,\emptyset}$ was obtained in~\cite[Theorem 4.8(a)]{FT1}.
To state the result, let $\Psi^\geq\colon \ddot{U}^{'\geq}\iso S^{\geq}$ be the natural extension of
the isomorphism $\Psi\colon \ddot{U}^{'+}\iso S$ from Theorem~\ref{Negut theorem}.

\begin{thm}\label{Theorem 3.8 from FT2}
The image of $L^{p,\bar{c}}_{\emptyset,\emptyset}$ under $\Psi^\geq$ has the following form:
  $$\Psi^\geq(L^{p,\bar{c}}_{\emptyset,\emptyset})=
    \sum_{N=0}^\infty a_{p,N} \C^{-N} q^{-d_1}q^{\bar{\Lambda}_p} \Gamma^0_{p,N},$$
where $a_{p,0}=1, a_{p,N}\in \CC[q^{\pm 1},d^{\pm 1}]$ and the shuffle elements $\Gamma^0_{p,N}\in S_{(N,\ldots,N)}$ are defined via
  $$\Gamma^0_{p,N}=
    \frac{\prod_{i\in [n]} \prod_{1\leq r\ne r'\leq N} (x_{i,r}-q^{-2}x_{i,r'})\cdot \prod_{i\in [n]}\prod_{r=1}^N x_{i,r}}
    {\prod_{i\in [n]} \prod_{1\leq r,r'\leq N} (x_{i,r}-x_{i+1,r'})}\cdot \prod_{r=1}^N \frac{x_{0,r}}{x_{p,r}}.$$
\end{thm}

Recall the Hopf pairing $'\varphi\colon {^{'}\ddot{U}^{\geq}}\times {^{'}\ddot{U}^{\leq}}\to \CC$
from Theorem~\ref{Drinfeld double sln}(d).
Clearly, the generators $h_{j,r}\ (r\in \ZZ_{>0})$ are orthogonal to all
generators of ${^{'}\ddot{U}^{\geq}}$ except for $h_{i,-r}$.
Moreover, we have
\begin{equation*}
  '\varphi(h_{i,-r}, h_{j,r})=\frac{[ra_{i,j}]_qd^{rm_{i,j}}}{r(q-q^{-1})}.
\end{equation*}
Note that the matrix $([ra_{i,j}]_qd^{rm_{i,j}})_{i,j\in [n]}$ is nondegenerate if $q,qd,qd^{-1}$ are not roots of unity.

\begin{defn}
 Let $\{h^{\perp}_{i,r}\}_{i\in [n]}$ be the basis of $\mathrm{span}_{\CC} \langle h_{0,-r},\cdots,h_{n-1,-r}\rangle$,
which is dual to $\{h_{i,r}\}_{i\in [n]}$ with respect to $'\varphi$.
In other words, $'\varphi(h^\perp_{i,r}, h_{j,s})=\delta_{i,j}\delta_{r,s}$ for any $i,j\in [n], r,s\in \ZZ_{>0}$.
\end{defn}

 Our next result provides the first insight towards the elements $\Psi^{-1}(\Gamma^0_{p,N})$.

\begin{lem}\label{shuffle realization of h_{i,-1}}
 We have $\Psi^{-1}(\Gamma^0_{p,1})=-(q^{-1}-q)^{-n}\varpi(h_{p,1}^\perp).$
\end{lem}

\begin{proof}[Proof of Lemma~\ref{shuffle realization of h_{i,-1}}]
$\ $

 Applying $\Psi$ to the formulas for $\varpi(h_{k,-1})$ of Proposition~\ref{explicit formulas for Miki}(b,c), we find:
  $$\Psi(\varpi(h_{k,-1}))=(q^{-1}-q)^{n-1}\cdot \frac{\prod_{i\in [n]} x_i}{\prod_{i\in [n]} (x_i-x_{i+1})}
    \cdot \left\{(q+q^{-1})\frac{x_0}{x_k}-d^{-1}\frac{x_0}{x_{k+1}}-d\frac{x_0}{x_{k-1}}\right\}.$$
Rewriting this as
 $\Psi(\varpi(h_{k,-1}))=-(q^{-1}-q)^{n} \sum_{p\in [n]} {'\varphi}(h_{k,-1},h_{p,1})\Gamma^0_{p,1}$,
we get the claim.
\end{proof}

 Now we are ready to state the main result of this section.

\begin{thm}\label{main5}
  Given $0\leq p\leq n-1,\ \bar{c}\in (\CC^\times)^{[n]}$, define $u\in \CC^\times$ as in Theorem~\ref{main4} via
\begin{equation*}\label{u-expression}
  u:=(-1)^{\frac{(n-2)(n-3)}{2}}q^{-1}d^{-p-(n-1)\delta_{p,0}}\C^{-1}\ \mathrm{with}\ \C=c_0\cdots c_{n-1}.
\end{equation*}

\noindent
(a) For any $i\ne p$, we have
  $$\varpi(e_i(z))\cdot L^{p,\bar{c}}_{\emptyset,\emptyset}=L^{p,\bar{c}}_{\emptyset,\emptyset}\cdot \varpi(e_i(z)),$$
  $$\varpi(f_i(z))\cdot L^{p,\bar{c}}_{\emptyset,\emptyset}=L^{p,\bar{c}}_{\emptyset,\emptyset}\cdot \varpi(f_i(z)),$$
  $$\varpi(\psi^\pm_i(z))\cdot L^{p,\bar{c}}_{\emptyset,\emptyset}=
    L^{p,\bar{c}}_{\emptyset,\emptyset}\cdot \varpi(\psi^\pm_i(z)).$$

\noindent
(b) We have
  $$(z-u)\cdot \varpi(e_p(z))\cdot L^{p,\bar{c}}_{\emptyset,\emptyset}=
    L^{p,\bar{c}}_{\emptyset,\emptyset}\cdot \varpi(e_p(z))\cdot (q^{-1}z-qu),$$
  $$(q^{-1}z-qu)\cdot \varpi(f_p(z))\cdot L^{p,\bar{c}}_{\emptyset,\emptyset}=
    L^{p,\bar{c}}_{\emptyset,\emptyset}\cdot \varpi(f_p(z))\cdot (z-u),$$
  $$\varpi(\psi^\pm_p(z))\cdot L^{p,\bar{c}}_{\emptyset,\emptyset}=
    L^{p,\bar{c}}_{\emptyset,\emptyset}\cdot \varpi(\psi^\pm_p(z)).$$

\noindent
(c) We have the following explicit formula
\begin{equation}\tag{$\sharp$}\label{sharp}
  L^{p,\bar{c}}_{\emptyset,\emptyset}=
  q^{-d_1}q^{\bar{\Lambda}_p}\exp \left(\sum_{r=1}^\infty \frac{[r]_q}{r}(qu)^r\varpi(h_{p,r}^\perp)\right).
\end{equation}
\end{thm}

\begin{rem}
 An analogous computation for the representation $\tau^{*,p}_{u}$ is much simpler.
The corresponding matrix element
 $L^{\tau^{*,p}_u}_{\emptyset,\emptyset}:=\langle 1\otimes \emptyset| (1\otimes \tau^{*,p}_u)({^{'}R}) |1\otimes \emptyset \rangle$
equals
  $$\langle 1\otimes \emptyset| (1\otimes \tau^{*,p}_{u})
    (q^{\frac{1}{n}(d_2-\sum_{j=1}^{n-1}\bar{\Lambda}_j)\otimes c'+c'\otimes \frac{1}{n}(d_2-\sum_{j=1}^{n-1}\bar{\Lambda}_j)+
    \sum_{j=1}^{n-1} \bar{\Lambda}_j\otimes h_{j,0}}
    \cdot \exp(\sum_{i\in[n]} \sum_{r=1}^\infty h_{i,r}^\perp\otimes h_{i,r})) |1\otimes \emptyset\rangle,$$
since $\tau^p_u({^{'}\ddot{U}^{-}})|\emptyset\rangle=0$. As
 $\tau^{*,p}_{u}(h_{i,r})|\emptyset\rangle^*=\delta_{i,p}\frac{[r]_qq^ru^r}{r}|\emptyset\rangle^* (r>0),
  \tau^{*,p}_{u}(h_{i,0})|\emptyset\rangle^*=\delta_{i,p}|\emptyset\rangle^*$,
we get
  $$L^{\tau^{*,p}_u}_{\emptyset,\emptyset}=q^{\frac{1}{n}(d_2-\sum_{j=1}^{n-1} \bar{\Lambda}_j)+\bar{\Lambda}_p}
    \exp\left(\sum_{r=1}^\infty \frac{[r]_q}{r}q^ru^r h_{p,r}^\perp\right).$$
In particular, $\varpi(L^{\tau^{*,p}_u}_{\emptyset,\emptyset})$ coincides
with the right-hand side of ($\sharp$). However, we are not aware of the
conceptual reason for $L^{p,\bar{c}}_{\emptyset,\emptyset}=\varpi(L^{\tau^{*,p}_u}_{\emptyset,\emptyset})$
(though it would immediately imply Theorem~\ref{main5}).
\end{rem}


\subsection{Proof of Theorem~\ref{main5}}
$\ $

 Our proof is based on the equality
\begin{equation}\tag{$\star$}\label{star}
  \langle 1\otimes w| (1\otimes \rho^{p}_{\bar{c}})(R' \Delta(x))| 1\otimes v \rangle=
  \langle 1\otimes w| (1\otimes \rho^{p}_{\bar{c}})(\Delta^{\mathrm{op}}(x) R')| 1\otimes v \rangle
\end{equation}
for any  $x\in \ddot{U}^{'}_{q,d}(\ssl_n),\ v\in W(p)_n,\ w\in W(p)^*_n$.

\medskip
\noindent
\emph{Notation:}
 Given a collection of elements $\beta_1,\cdots,\beta_N\in \{\pm \bar{\alpha}_0,\cdots, \pm \bar{\alpha}_{n-1}\}$
and $0\leq p\leq n-1$, consider $v_0\otimes e^{\beta_1}\cdots e^{\beta_N}e^{\bar{\Lambda}_p}$--an element of $W(p)_n$.
We will also use the same notation for a dual element of $W(p)_n^*$,
when writing it in the matrix coefficients of $L$ operators.

\medskip
\noindent
$\bullet$ \emph{Case $p\ne 0$.}

\medskip
(a) We need to show that $L^{p,\bar{c}}_{\emptyset,\emptyset}$ commutes with
$\{\varpi(e_{i,k}), \varpi(f_{i,k}),\varpi(h_{i,k})\}_{i\ne p}^{k\in \ZZ}$.

\medskip
\noindent
$\circ$
\emph{Proof of $[L^{p,\bar{c}}_{\emptyset,\emptyset},\varpi(e_{i,0})]=0$ and
      $[L^{p,\bar{c}}_{\emptyset,\emptyset},\varpi(f_{i,0})]=0$ for $i\ne 0,p$.}

Due to~(\ref{coproduct}), we have
  $$\Delta(e_{i,k})=
    e_{i,k}\otimes 1+\psi_{i,0}^{-1}\gamma^{-k}\otimes e_{i,k}+\sum_{r>0} \psi_{i,-r}\gamma^{-k-r/2}\otimes  e_{i,k+r},$$
  $$\Delta(f_{i,k})=
    1\otimes f_{i,k}+f_{i,k}\otimes \psi_{i,0}\gamma^{-k}+\sum_{r>0} f_{i,k-r}\otimes \psi_{i,r}\gamma^{-k+r/2}.$$
Evaluating both sides of~(\ref{star})
at $v=|\emptyset\rangle, w=\langle\emptyset|$ and $x=e_{i,0}$ or $x=f_{i,0}$, we immediately get
$[L^{p,\bar{c}}_{\emptyset,\emptyset},e_{i,0}]=0$ and $[L^{p,\bar{c}}_{\emptyset,\emptyset},f_{i,0}]=0$.
It remains to use $\varpi(e_{i,0})=e_{i,0},\ \varpi(f_{i,0})=f_{i,0}$ for $i\ne 0$.
$\checkmark$

\medskip
\noindent
$\circ$
\emph{Proof of $[L^{p,\bar{c}}_{\emptyset,\emptyset},\varpi(e_{0,-1})]=0$ and
      $[L^{p,\bar{c}}_{\emptyset,\emptyset},\varpi(f_{0,1})]=0$.}

Evaluating both sides of~(\ref{star}) at
$v=|\emptyset\rangle, w=\langle\emptyset|$ and $x=e_{0,1}$ or $x=f_{0,-1}$, we immediately get
$[L^{p,\bar{c}}_{\emptyset,\emptyset},e_{0,1}]=0$ and $[L^{p,\bar{c}}_{\emptyset,\emptyset},f_{0,-1}]=0$, respectively.
It remains to apply the equalities
$\varpi(e_{0,-1})=(-d)^n e_{0,1}$ and $\varpi(f_{0,1})=(-d)^{-n} f_{0,-1}$
from Proposition~\ref{explicit formulas for Miki}(d).
$\checkmark$

\medskip
\noindent
$\circ$
\emph{Proof of $[L^{p,\bar{c}}_{\emptyset,\emptyset},\varpi(h_{i,-1})]=0$ for any $i\in [n]$.}

It suffices to prove
 $[\Psi^\geq(L^{p,\bar{c}}_{\emptyset,\emptyset}),\Psi^\geq(\varpi(h_{i,-1}))]=0$.
According to Lemma~\ref{shuffle realization of h_{i,-1}},
$\Psi(\varpi(h_{i,-1}))$ is a linear combination of $\Gamma^{0}_{p',1}$.
On the other hand, $\Psi^\geq(L^{p,\bar{c}}_{\emptyset,\emptyset})$
is a linear combination of $q^{-d_1}q^{\bar{\Lambda}_p}\Gamma^0_{p,N}$, due to Theorem~\ref{Theorem 3.8 from FT2}.
The commutativity of the elements $\{\Gamma^0_{p',m}\}^{m\in \NN}_{p'\in [n]}$ has been established in~\cite{FT1},
while $q^{-d_1}q^{\bar{\Lambda}_p}$ obviously commutes with $\Gamma^0_{p',1}$.
The result follows.
$\checkmark$

\medskip
\noindent
$\circ$
\emph{Proof of $[L^{p,\bar{c}}_{\emptyset,\emptyset},\varpi(h_{i,1})]=0$ for $i\ne 0,p$.}

 According to Proposition~\ref{explicit formulas for Miki}(b), it suffices to prove that $E=0$, where $E$ is defined via
\begin{equation}\label{nontrivial multicommutator 1}
  E:=[L^{p,\bar{c}}_{\emptyset,\emptyset};
      [f_{i,0},[f_{i-1,0},\cdots,[f_{1,0},[f_{i+1,0},\cdots,[f_{n-1,0},f_{0,0}]_{q^{-1}}\cdots ]_{q^{-1}}]_{q^{-1}}\cdots ]_{q^{-1}}]_{q^{-2}}]_1.
\end{equation}

In what follows, we assume $i<p\leq n-1$ leaving the case $0<p<i$ to the interested reader.
Applying iteratively the $q$-commutator identity (mentioned in our proof of Proposition~\ref{explicit formulas for Miki}(b))
\begin{equation}\tag{$\diamondsuit$}\label{diamondsuit}
  [a,[b,c]_u]_v=[[a,b]_x,c]_{uv/x}+x\cdot [b,[a,c]_{v/x}]_{u/x}
\end{equation}
together with $[L^{p,\bar{c}}_{\emptyset,\emptyset},f_{j,0}]=0$ for $j\ne 0,p$, we reduce
to a stronger equality $E^{(1)}_1+E^{(1)}_2=0$ with
  $$E^{(1)}_1:=[[L^{p,\bar{c}}_{\emptyset,\emptyset},f_{p,0}]_{q^{-1}},[f_{p+1,0},\cdots,[f_{n-1,0},f_{0,0}]_{q^{-1}}\cdots]_{q^{-1}}]_1,$$
  $$E^{(1)}_2:=q^{-1}[f_{p,0},[f_{p+1,0},\cdots,[f_{n-1,0},[L^{p,\bar{c}}_{\emptyset,\emptyset},f_{0,0}]_q]_{q^{-1}}\cdots]_{q^{-1}}]_1.$$
Evaluating both sides of~(\ref{star}) for appropriate $v,w,x$ step-by-step, we obtain an explicit formula
  $$E^{(1)}_2=-(-q)^{p-n}\cdot \frac{\psi_{p+1,0}\cdots \psi_{n-1,0}\psi_{0,0}}{c_p\cdots c_{n-1}c_0}\cdot
    L^{p,\bar{c}}_{v_0\otimes e^{\bar{\alpha}_{p+1}}\cdots e^{\bar{\alpha}_{n-1}}e^{\bar{\alpha}_{0}} e^{\bar{\Lambda}_p}, v_0\otimes e^{-\bar{\alpha}_p}e^{\bar{\Lambda}_p}}.$$

Let us now compute $E^{(1)}_1$.
Evaluating both sides of~(\ref{star}) at $v=|\emptyset \rangle,\ w=\langle \emptyset|,\ x=f_{p,0}$, we find
  $$[L^{p,\bar{c}}_{\emptyset,\emptyset},f_{p,0}]_{q^{-1}}=
    -q^{-1}c_p^{-1}\cdot L^{p,\bar{c}}_{\emptyset, v_0\otimes e^{-\bar{\alpha}_p}e^{\bar{\Lambda}_p}}.$$
Evaluating both sides of~(\ref{star}) at $v=v_0\otimes e^{-\bar{\alpha}_p}e^{\bar{\Lambda}_p}, w=\langle \emptyset|, x=f_{j,0}$,
we find $[L^{p,\bar{c}}_{\emptyset, v_0\otimes  e^{-\bar{\alpha}_p}e^{\bar{\Lambda}_p}},f_{j,0}]=0$ for $p+1<j\leq  n-1$.
Applying iteratively~(\ref{diamondsuit}), we get $E^{(1)}_1=E^{(2)}_1+E^{(2)}_2$ with
  $$E^{(2)}_1:=-q^{-1}c_p^{-1}
    [[L^{p,\bar{c}}_{\emptyset,v_0\otimes
    e^{-\bar{\alpha}_p}e^{\bar{\Lambda}_p}},f_{p+1,0}]_{q^{-1}},[f_{p+2,0},\cdots,[f_{n-1,0},f_{0,0}]_{q^{-1}}\cdots]_{q^{-1}}]_1,$$
  $$E^{(2)}_2:=-q^{-2}c_p^{-1}
    [f_{p+1,0},[f_{p+2,0},\cdots,[f_{n-1,0},[L^{p,\bar{c}}_{\emptyset,v_0\otimes
    e^{-\bar{\alpha}_p}e^{\bar{\Lambda}_p}},f_{0,0}]_q]_{q^{-1}}\cdots]_{q^{-1}}]_1.$$
Evaluating both sides of~(\ref{star}) for appropriate $v,w,x$ step-by-step, we obtain an explicit formula
\begin{multline*}
  E^{(2)}_2=
  (-q)^{p-n}\cdot \frac{\psi_{p+1,0}\cdots \psi_{n-1,0}\psi_{0,0}}{c_p\cdots c_{n-1}c_0}\cdot
  L^{p,\bar{c}}_{v_0\otimes e^{\bar{\alpha}_{p+1}}\cdots e^{\bar{\alpha}_{n-1}}e^{\bar{\alpha}_{0}} e^{\bar{\Lambda}_p},
  v_0\otimes  e^{-\bar{\alpha}_p}e^{\bar{\Lambda}_p}}\\
  -(-q)^{p-n}\cdot \frac{\psi_{p+2,0}\cdots \psi_{n-1,0}\psi_{0,0}}{c_p\cdots c_{n-1}c_0}\cdot
  L^{p,\bar{c}}_{v_0\otimes e^{\bar{\alpha}_{p+2}}\cdots e^{\bar{\alpha}_{n-1}}e^{\bar{\alpha}_{0}} e^{\bar{\Lambda}_p},
  v_0\otimes e^{-\bar{\alpha}_{p+1}}e^{-\bar{\alpha}_p}e^{\bar{\Lambda}_p}}.
\end{multline*}
The first summand cancels $E^{(1)}_2$, while the second summand is very similar to $E^{(1)}_2$.

Evaluating $E^{(2)}_1$, we get a similar formula $E^{(2)}_1=E^{(3)}_1+E^{(3)}_2$ with
\begin{multline*}
  E^{(3)}_2=
  (-q)^{p-n}\cdot \frac{\psi_{p+2,0}\cdots \psi_{n-1,0}\psi_{0,0}}{c_p\cdots c_{n-1}c_0}\cdot
  L^{p,\bar{c}}_{v_0\otimes e^{\bar{\alpha}_{p+2}}\cdots e^{\bar{\alpha}_{n-1}}e^{\bar{\alpha}_{0}} e^{\bar{\Lambda}_p},
  v_0\otimes e^{-\bar{\alpha}_{p+1}}e^{-\bar{\alpha}_p}e^{\bar{\Lambda}_p}}\\
  -(-q)^{p-n}\cdot \frac{\psi_{p+3,0}\cdots \psi_{n-1,0}\psi_{0,0}}{c_p\cdots c_{n-1}c_0}\cdot
  L^{p,\bar{c}}_{v_0\otimes e^{\bar{\alpha}_{p+3}}\cdots e^{\bar{\alpha}_{n-1}}e^{\bar{\alpha}_{0}} e^{\bar{\Lambda}_p},
  v_0\otimes e^{-\bar{\alpha}_{p+2}}e^{-\bar{\alpha}_{p+1}}e^{-\bar{\alpha}_p}e^{\bar{\Lambda}_p}}.
\end{multline*}
Proceeding further in the same way, we see that all nontrivial summands
in the formula for $E$ split into pairs of opposite terms.
Hence, $E=0$ and so $[L^{p,\bar{c}}_{\emptyset,\emptyset}, \varpi(h_{i,1})]=0$ for $i\ne 0,p$.
$\checkmark$

\medskip
\noindent
$\circ$
\emph{Proof of $[L^{p,\bar{c}}_{\emptyset,\emptyset},\varpi(e_{i,k})]=0$ and
      $[L^{p,\bar{c}}_{\emptyset,\emptyset},\varpi(f_{i,k})]=0$ for $i\ne p$ and any $k\in \ZZ$.}

Choose $j\ne 0,p$ such that $a_{j,i}\ne 0$.
Combining the commutator relations
  $$[\varpi(h_{j,\pm 1}),\varpi(e_{i,k})]=d^{\mp m_{j,i}}[a_{j,i}]_q\cdot \varpi(e_{i,k\pm 1}),\
    [\varpi(h_{j,\pm 1}),\varpi(f_{i,k})]=-d^{\mp m_{j,i}}[a_{j,i}]_q\cdot \varpi(f_{i,k\pm 1})$$
with 
  $$[L^{p,\bar{c}}_{\emptyset,\emptyset},\varpi(e_{i,-\delta_{i,0}})]=0,\
    [L^{p,\bar{c}}_{\emptyset,\emptyset},\varpi(f_{i,\delta_{i,0}})]=0,\
    [L^{p,\bar{c}}_{\emptyset,\emptyset},\varpi(h_{j,\pm 1})]=0$$
established above, we get 
$[L^{p,\bar{c}}_{\emptyset,\emptyset},\varpi(e_{i,k})]=0$ and $[L^{p,\bar{c}}_{\emptyset,\emptyset},\varpi(f_{i,k})]=0$
by induction on $k$.
$\checkmark$

\medskip
\noindent
$\circ$
\emph{Proof of $[L^{p,\bar{c}}_{\emptyset,\emptyset},\varpi(\psi_{i,k})]=0$ for $i\ne p$ and any $k\in \ZZ$.}

For $k\ne 0$, this follows immediately from the defining relation (T4) and the previous step.
For $k=0$, it suffices to prove $[L^{p,\bar{c}}_{\emptyset,\emptyset},\psi_{i,0}]=0$ for any $i\in  [n]$.
This equality follows by evaluating both sides of~(\ref{star}) at $w=\langle \emptyset|,\ v=|\emptyset \rangle,\ x=\psi_{i,0}$.
$\checkmark$

\medskip
(b) The first two equalities of (b) are equivalent to the following identities:
\begin{equation}\label{ii_1}
  [L^{p,\bar{c}}_{\emptyset,\emptyset},\varpi(e_{p,k+1})]_q=
  q^2u\cdot [L^{p,\bar{c}}_{\emptyset,\emptyset},\varpi(e_{p,k})]_{q^{-1}}\ \ \forall\ k\in \ZZ,
\end{equation}
\begin{equation}\label{ii_2}
  [L^{p,\bar{c}}_{\emptyset,\emptyset},\varpi(f_{p,k-1})]_q=
  u^{-1}\cdot [L^{p,\bar{c}}_{\emptyset,\emptyset},\varpi(f_{p,k})]_{q^{-1}}\ \ \forall\ k\in \ZZ.
\end{equation}
It suffices to check~(\ref{ii_1}) and~(\ref{ii_2}) for single values of $k$ as we can derive
the general equalities by commuting further iteratively with $\varpi(h_{p+1,\pm 1})$.

\medskip
\noindent
$\circ$
\emph{Proof of $[L^{p,\bar{c}}_{\emptyset,\emptyset},\varpi(e_{p,1})]_q=q^2u\cdot [L^{p,\bar{c}}_{\emptyset,\emptyset},\varpi(e_{p,0})]_{q^{-1}}$.}

Evaluating both sides of~(\ref{star}) at $v=|\emptyset\rangle,\ w=\langle \emptyset|,\ x=e_{p,0}$, we find
  $[L^{p,\bar{c}}_{\emptyset,\emptyset},\varpi(e_{p,0})]_{q^{-1}}=
   c_p\cdot L^{p,\bar{c}}_{v_0\otimes e^{-\bar{\alpha}_p}e^{\bar{\Lambda}_p},\emptyset}$.
As $\varpi(e_{p,0})=e_{p,0}$, we finally get
  $$q^2u\cdot [L^{p,\bar{c}}_{\emptyset,\emptyset},\varpi(e_{p,0})]_{q^{-1}}=
    q^2c_pu\cdot L^{p,\bar{c}}_{v_0\otimes e^{-\bar{\alpha}_p}e^{\bar{\Lambda}_p},\emptyset}.$$

To compute $[L^{p,\bar{c}}_{\emptyset,\emptyset},\varpi(e_{p,1})]_q$,
let us first evaluate $\varpi(e_{p,1})$.
Due to~(\ref{T5'}), we have
  $$[h_{p,1},e_{p,0}]=(q+q^{-1})e_{p,1}\Rightarrow \varpi(e_{p,1})=-(q+q^{-1})^{-1}\cdot [\varpi(e_{p,0}),\varpi(h_{p,1})],$$
where $\varpi(e_{p,0})=e_{p,0}$ and
  $$\varpi(h_{p,1})=(-1)^{n+p}d^{-p}q^n\cdot
    [f_{p,0},\cdots,[f_{1,0},[f_{p+1,0},\cdots,[f_{n-1,0},f_{0,0}]_{q^{-1}}\cdots ]_{q^{-1}}]_{q^{-1}}\cdots]_{q^{-2}}.$$
Applying iteratively the equality~(\ref{diamondsuit}) together with the relation~(\ref{T4}), we finally get
  $$\varpi(e_{p,1})=
    (-1)^{n+p+1}d^{-p}q^{n-2}\psi_{p,0}
    [f_{p-1,0},\cdots,[f_{1,0},[f_{p+1,0},\cdots,[f_{n-1,0},f_{0,0}]_{q^{-1}}\cdots ]_{q^{-1}}]_{q^{-1}}\cdots]_{q^{-1}}$$
Therefore, it remains to evaluate
  $$E:=[L^{p,\bar{c}}_{\emptyset,\emptyset},
        [f_{p-1,0},\cdots,[f_{1,0},[f_{p+1,0},\cdots,[f_{n-1,0},f_{0,0}]_{q^{-1}}\cdots ]_{q^{-1}}]_{q^{-1}}\cdots]_{q^{-1}}]_q.$$
Applying iteratively the equality~(\ref{diamondsuit}) together with
$[L^{p,\bar{c}}_{\emptyset,\emptyset},f_{j,0}]=0$ for $j\ne 0,p$, we get
  $$E=[f_{p-1,0},\cdots,[f_{1,0},[f_{p+1,0},\cdots,[f_{n-1,0},[L^{p,\bar{c}}_{\emptyset,\emptyset},f_{0,0}]_q]_{q^{-1}}\cdots ]_{q^{-1}}]_{q^{-1}}\cdots]_{q^{-1}}.$$
To compute this multicommutator, we apply the equality~(\ref{star}) with an appropriate choice of $v,w,x$ step-by-step.
Leaving details to the interested reader, let us present the final formula
  $$E=(-1)^nq^{3-n}\prod_{i\ne p}\frac{\psi_{i,0}}{c_i}\cdot
    L^{p,\bar{c}}_
    {v_0\otimes e^{\bar{\alpha}_{p-1}}\cdots e^{\bar{\alpha}_1}e^{\bar{\alpha}_{p+1}}\cdots e^{\bar{\alpha}_{n-1}}e^{\bar{\alpha}_0}e^{\bar{\Lambda}_p}, \emptyset}.$$
Since $\prod_{i\in [n]} \psi_{i,0}=1$ in $\ddot{U}^{'}_{q,d}(\ssl_n)$, we finally get
  $$[L^{p,\bar{c}}_{\emptyset,\emptyset}, \varpi(e_{p,1})]_q=
    (-1)^{\frac{(n-2)(n-3)}{2}}d^{-p}qc_p\C^{-1}\cdot L^{p,\bar{c}}_{v_0\otimes e^{-\bar{\alpha}_p}e^{\bar{\Lambda}_p},\emptyset},$$
where we used the following identity in $\CC\{\bar{P}\}$
  $$e^{\bar{\alpha}_{p-1}}\cdots e^{\bar{\alpha}_1}e^{\bar{\alpha}_{p+1}}\cdots e^{\bar{\alpha}_{n-1}}e^{\bar{\alpha}_0}=
    (-1)^{\frac{n(n-1)}{2}+p}e^{-\bar{\alpha}_p}.$$

The equality
 $[L^{p,\bar{c}}_{\emptyset,\emptyset},\varpi(e_{p,1})]_q=q^2u\cdot [L^{p,\bar{c}}_{\emptyset,\emptyset},\varpi(e_{p,0})]_{q^{-1}}$
follows.
$\checkmark$

\medskip
\noindent
$\circ$
\emph{Proof of $[L^{p,\bar{c}}_{\emptyset,\emptyset},\varpi(f_{p,-1})]_q=u^{-1}\cdot [L^{p,\bar{c}}_{\emptyset,\emptyset},\varpi(f_{p,0})]_{q^{-1}}$.}

Evaluating both sides of~(\ref{star}) at $v=|\emptyset\rangle,\ w=\langle \emptyset|,\ x=f_{p,0}$, we find
 $[L^{p,\bar{c}}_{\emptyset,\emptyset},\varpi(f_{p,0})]_{q^{-1}}=
  \frac{-1}{qc_p}\cdot L_{\emptyset, v_0\otimes e^{-\bar{\alpha}_p}e^{\bar{\Lambda}_p}}$.
As $\varpi(f_{p,0})=f_{p,0}$, we finally get
  $$u^{-1}\cdot [L^{p,\bar{c}}_{\emptyset,\emptyset},\varpi(f_{p,0})]_{q^{-1}}=
    -q^{-1}c^{-1}_pu^{-1}\cdot L^{p,\bar{c}}_{\emptyset, v_0\otimes e^{-\bar{\alpha}_p}e^{\bar{\Lambda}_p}}.$$

To evaluate $[L^{p,\bar{c}}_{\emptyset,\emptyset},\varpi(f_{p,-1})]_q$,
let us first compute $\varpi(f_{p,-1})$.
Due to~(\ref{T6'}), we have
  $$[h_{p,-1},f_{p,0}]=-(q+q^{-1})f_{p,-1}\Rightarrow \varpi(f_{p,-1})=(q+q^{-1})^{-1}\cdot [\varpi(f_{p,0}),\varpi(h_{p,-1})],$$
where $\varpi(f_{p,0})=f_{p,0}$ and
  $$\varpi(h_{p,-1})=(-1)^{p+1}d^p\cdot
    [e_{p,0},\cdots,[e_{1,0},[e_{p+1,0},\cdots,[e_{n-1,0},e_{0,0}]_{q^{-1}}\cdots ]_{q^{-1}}]_{q^{-1}}\cdots]_{q^{-2}}.$$
Applying iteratively the equality~(\ref{diamondsuit}) together with the relation~(\ref{T4}), we finally get
  $$\varpi(f_{p,-1})=(-1)^{p+1}d^p\cdot
    [e_{p-1,0},\cdots,[e_{1,0},[e_{p+1,0},\cdots,[e_{n-1,0},e_{0,0}]_{q^{-1}}\cdots ]_{q^{-1}}]_{q^{-1}}\cdots]_{q^{-1}}\cdot \psi^{-1}_{p,0}.$$
Therefore, it remains to evaluate
  $$E:=[L^{p,\bar{c}}_{\emptyset,\emptyset},
        [e_{p-1,0},\cdots,[e_{1,0},[e_{p+1,0},\cdots,[e_{n-1,0},e_{0,0}]_{q^{-1}}\cdots ]_{q^{-1}}]_{q^{-1}}\cdots]_{q^{-1}}]_q.$$
Applying iteratively the equality~(\ref{diamondsuit}) together with $[L^{p,\bar{c}}_{\emptyset,\emptyset},e_{j,0}]=0$ for $j\ne 0,p$, we get
  $$E=[e_{p-1,0},\cdots,[e_{1,0},[e_{p+1,0},\cdots,[e_{n-1,0},[L^{p,\bar{c}}_{\emptyset,\emptyset},e_{0,0}]_q]_{q^{-1}}\cdots ]_{q^{-1}}]_{q^{-1}}\cdots]_{q^{-1}}.$$
To compute this multicommutator, we apply the equality~(\ref{star}) with an appropriate choice of $v,w,x$ step-by-step.
Leaving details to the interested reader, let us present the final formula
  $$E=-c^{-1}_p\C\cdot
    L^{p,\bar{c}}_
    {\emptyset,v_0\otimes e^{\bar{\alpha}_{p-1}}\cdots e^{\bar{\alpha}_1}e^{\bar{\alpha}_{p+1}}\cdots e^{\bar{\alpha}_{n-1}}e^{\bar{\alpha}_0}e^{\bar{\Lambda}_p}}\cdot \psi_{p,0}.$$
Therefore,
  $$[L^{p,\bar{c}}_{\emptyset,\emptyset}, \varpi(f_{p,-1})]_q=
    (-1)^{\frac{n(n-1)}{2}}d^pc_p^{-1}\C\cdot L^{p,\bar{c}}_{\emptyset,v_0\otimes e^{-\bar{\alpha}_p}e^{\bar{\Lambda}_p}}.$$

The equality
 $[L^{p,\bar{c}}_{\emptyset,\emptyset},\varpi(f_{p,-1})]_q=u^{-1}\cdot [L^{p,\bar{c}}_{\emptyset,\emptyset},\varpi(f_{p,0})]_{q^{-1}}$
follows.
$\checkmark$

\medskip
\noindent
$\circ$
\emph{Proof of $[L^{p,\bar{c}}_{\emptyset,\emptyset},\varpi(\psi^{\pm}_p(z))]=0$.}

Define $\wt{\psi}_{p,N}\in {^{'}\ddot{U}}_{q,d}(\ssl_n)$ as the coefficient of $z^{-N}$ in $\psi^+_p(z)-\psi^-_p(z)$,
so that $[e_{p,a},f_{p,b}]=\frac{\wt{\psi}_{p,a+b}}{q-q^{-1}}$ for any $a,b\in \ZZ$.
Set $X_N:=[\varpi(\wt{\psi}_{p,N}),L^{p,\bar{c}}_{\emptyset,\emptyset}]$.
Combining the equalities
  $$(\varpi(e_{p,k+1})-u\varpi(e_{p,k}))L^{p,\bar{c}}_{\emptyset,\emptyset}=
    L^{p,\bar{c}}_{\emptyset,\emptyset}(q^{-1}\varpi(e_{p,k+1})-qu\varpi(e_{p,k})),$$
  $$(q^{-1}\varpi(f_{p,l+1})-qu\varpi(f_{p,l}))L^{p,\bar{c}}_{\emptyset,\emptyset}=
    L^{p,\bar{c}}_{\emptyset,\emptyset}(\varpi(f_{p,l+1})-u\varpi(f_{p,l})),$$
we get the following recursive relation:
$q^{-1}X_{k+l+2}-u(q+q^{-1})X_{k+l+1}+u^2qX_{k+l}=0.$

As $X_{-1}=X_0=0$, we get $X_k=0$ for any $k\in \ZZ$.
This proves $[L^{p,\bar{c}}_{\emptyset,\emptyset},\varpi(\psi^\pm_p(z))]=0$.
$\checkmark$

\medskip
(c) The unique element satisfying conditions (a,b) of Theorem~\ref{main5} and whose shuffle interpretation
has a form as in Theorem~\ref{Theorem 3.8 from FT2}
(we only need to know that it lives in an appropriate completion and its `purely Cartan part' equals $q^{-d_1}q^{\bar{\Lambda}_p}$)
is given by the right-hand side of~(\ref{sharp}).


\medskip
\noindent
$\bullet$ \emph{Case $p=0$.}

Parts (a) and (c) are proved completely analogously to the case $p\ne 0$.
Since the last equality in (b) follows from the former two,
it suffices to check~(\ref{ii_1}) and~(\ref{ii_2}) for some $k\in \ZZ.$

\medskip
\noindent
$\circ$
\emph{Proof of $[L^{0,\bar{c}}_{\emptyset,\emptyset},\varpi(e_{0,0})]_q=q^2u\cdot [L^{0,\bar{c}}_{\emptyset,\emptyset},\varpi(e_{0,-1})]_{q^{-1}}$.}

According to Proposition~\ref{explicit formulas for Miki}(d), we have $\varpi(e_{0,-1})=(-d)^n e_{0,1}$.
Evaluating both sides of~(\ref{star}) at $v=|\emptyset\rangle,\ w=\langle \emptyset|,\ x=e_{0,1}$, we get
 $[L^{0,\bar{c}}_{\emptyset,\emptyset},e_{0,1}]_{q^{-1}}=(-1)^nc_0\cdot L^{0,\bar{c}}_{v_0\otimes e^{-\bar{\alpha}_0},\emptyset}.$
Therefore
  $$q^2u\cdot [L^{0,\bar{c}}_{\emptyset,\emptyset},\varpi(e_{0,-1})]_{q^{-1}}=uq^2d^nc_0\cdot
    L^{0,\bar{c}}_{v_0\otimes e^{-\bar{\alpha}_0},\emptyset}.$$
Next, we evaluate the left-hand side of the claimed equality. According to Proposition~\ref{explicit formulas for Miki}(a)
  $$\varpi(e_{0,0})=d (-q)^{n-2}\gamma \psi_{0,0}\cdot [f_{n-1,0},\cdots,[f_{2,0},f_{1,1}]_{q^{-1}}\cdots ]_{q^{-1}}.$$
Applying iteratively~(\ref{diamondsuit}) together with $[L^{0,\bar{c}}_{\emptyset,\emptyset},\psi_{0,0}]=0$ 
and $[L^{0,\bar{c}}_{\emptyset,\emptyset}, f_{j,0}]=0$ for $j\ne 0$, we get
  $$[L^{0,\bar{c}}_{\emptyset,\emptyset}, \varpi(e_{0,0})]_q=d(-q)^{n-2}\gamma\psi_{0,0}\cdot
    [f_{n-1,0},\cdots, [f_{2,0},[L^{0,\bar{c}}_{\emptyset,\emptyset},f_{1,1}]_q]_{q^{-1}}\cdots ]_{q^{-1}}.$$
Evaluating this multicommutator step-by-step as before, we finally get
  $$[L^{0,\bar{c}}_{\emptyset,\emptyset}, \varpi(e_{0,0})]_q=
    qdc_0\C^{-1}\cdot L^{0,\bar{c}}_{v_0\otimes e^{\bar{\alpha}_{n-1}}\cdots e^{\bar{\alpha}_1},\emptyset}=
    (-1)^{\frac{(n-2)(n-3)}{2}}qdc_0\C^{-1}\cdot L^{0,\bar{c}}_{v_0\otimes e^{-\bar{\alpha}_0},\emptyset}.$$

The equality
 $[L^{0,\bar{c}}_{\emptyset,\emptyset},\varpi(e_{0,0})]_q=q^2u\cdot [L^{0,\bar{c}}_{\emptyset,\emptyset},\varpi(e_{0,-1})]_{q^{-1}}$
follows.
$\checkmark$

\medskip
\noindent
$\circ$ \emph{Proof of $[L^{0,\bar{c}}_{\emptyset,\emptyset},\varpi(f_{0,0})]_q=u^{-1}[L^{0,\bar{c}}_{\emptyset,\emptyset},\varpi(f_{0,1})]_{q^{-1}}$.}

According to Proposition~\ref{explicit formulas for Miki}(d), we have $\varpi(f_{0,1})=(-d)^{-n} f_{0,-1}$.
Evaluating both sides of~(\ref{star}) at $v=|\emptyset\rangle,\ w=\langle \emptyset|,\ x=f_{0,-1}$, we get
 $[L^{0,\bar{c}}_{\emptyset,\emptyset},f_{0,-1}]_{q^{-1}}=\frac{(-1)^{n+1}}{qc_0}\cdot L^{0,\bar{c}}_{\emptyset,v_0\otimes e^{-\bar{\alpha}_0}}.$
Hence
  $$u^{-1}\cdot [L^{0,\bar{c}}_{\emptyset,\emptyset},\varpi(f_{0,1})]_{q^{-1}}=
    -q^{-1}d^{-n}c_0^{-1}u^{-1}\cdot L^{0,\bar{c}}_{\emptyset,v_0\otimes e^{-\bar{\alpha}_0}}.$$
Let us now evaluate the left-hand side of the claimed equality. According to Proposition~\ref{explicit formulas for Miki}(a)
  $$\varpi(f_{0,0})=d^{-1}\cdot [e_{n-1,0},\cdots,[e_{2,0},e_{1,-1}]_{q^{-1}}\cdots ]_{q^{-1}}\cdot \psi^{-1}_{0,0}\gamma^{-1}.$$
Applying iteratively~(\ref{diamondsuit}) together with $[L^{0,\bar{c}}_{\emptyset,\emptyset},\psi_{0,0}]=0$ 
and $[L^{0,\bar{c}}_{\emptyset,\emptyset}, e_{j,0}]=0$ for $j\ne 0$, we get
  $$[L^{0,\bar{c}}_{\emptyset,\emptyset}, \varpi(f_{0,0})]_q=d^{-1}\cdot
    [e_{n-1,0},\cdots [e_{2,0},[L^{0,\bar{c}}_{\emptyset,\emptyset},e_{1,-1}]_q]_{q^{-1}}\cdots ]_{q^{-1}}\cdot \psi^{-1}_{0,0}\gamma^{-1}.$$
Evaluating this multicommutator step-by-step as before, we finally get
  $$[L^{0,\bar{c}}_{\emptyset,\emptyset}, \varpi(f_{0,0})]_q=
    -d^{-1}c_0^{-1}\C\cdot L^{0,\bar{c}}_{\emptyset, v_0\otimes e^{\bar{\alpha}_{n-1}}\cdots e^{\bar{\alpha}_1}}=
    -(-1)^{\frac{(n-2)(n-3)}{2}}d^{-1}c_0^{-1}\C\cdot L^{0,\bar{c}}_{\emptyset, v_0\otimes e^{-\bar{\alpha}_0}}.$$

The equality
 $[L^{0,\bar{c}}_{\emptyset,\emptyset},\varpi(f_{0,0})]_q=u^{-1}[L^{0,\bar{c}}_{\emptyset,\emptyset},\varpi(f_{0,1})]_{q^{-1}}$
follows.
$\checkmark$

\medskip
\noindent
This completes our proof of Theorem~\ref{main5} for any $p\in [n]$.


\subsection{Bimodule $\SSS(p,\bar{c})$}
$\ $

Let $^{'}\ddot{U}^{\geq,\wedge}$ (resp. $^{'}\ddot{U}^{+,\wedge}$) be the completion of $^{'}\ddot{U}^{\geq}$
(resp. $^{'}\ddot{U}^{+}$) with respect to the $\ZZ$--grading on $^{'}\ddot{U}^{\geq}$ (resp. $^{'}\ddot{U}^{+}$)
defined by assigning $\deg(e_{i,k})=-k, \deg(h_{i,k})=-k, \deg(q^{d_2})=0$.
Note that
 $L^{p,\bar{c}}_{\emptyset,\emptyset}\in \varpi(^{'}\ddot{U}^{\geq,\wedge})$,
due to Theorem~\ref{main5}.
Consider the $^{'}\ddot{U}^{+}$-bimodule $\SSS(p,\bar{c})$ defined as
  $$\SSS(p,\bar{c}):=
    \varpi({^{'}\ddot{U}^{+}})\cdot L^{p,\bar{c}}_{\emptyset,\emptyset}\cdot \varpi({^{'}\ddot{U}^{+}})
    \subset \varpi(^{'}\ddot{U}^{\geq,\wedge}),$$
where both $^{'}\ddot{U}^{+}$-actions are in conjunction with $\varpi$.
We conclude this section with the following result analogous to~\cite[Lemma 4.4]{FJMM2}.

\begin{prop}\label{L-opertaor realization}
There exists an isomorphism of $^{'}\ddot{U}^{+}$-bimodules
  $$\iota\colon S_{1,p}(u)\iso \SSS(p,\bar{c})\ \mathrm{with}\ \boldsymbol{1}\mapsto L^{p,\bar{c}}_{\emptyset,\emptyset},$$
where
$u=(-1)^{\frac{(n-2)(n-3)}{2}}q^{-1}d^{-p-(n-1)\delta_{p,0}}(c_0\cdots c_{n-1})^{-1}$ as before.
\end{prop}

\begin{proof}[Proof of Proposition~\ref{L-opertaor realization}]
$\ $

Any element $H\in S_{1,p}(u)$ can be written as $H=\sum_{l} F_l\star \textbf{1}\star G_l$ with $F_l,G_l\in S$,
due to Theorem~\ref{main1}.
Set $\iota(H):=\sum_l \varpi(a_l)\cdot L^{p,\bar{c}}_{\emptyset, \emptyset}\cdot \varpi(b_l)$,
where $a_l:=\Psi^{-1}(F_l), b_l:=\Psi^{-1}(G_l) \in {^{'}\ddot{U}^+}$.
We must show that $\iota$ is well-defined.
Applying Theorem~\ref{main5}(a,b), we find
  $$\varpi(e_{i,k})\cdot L^{p,\bar{c}}_{\emptyset,\emptyset}=L^{p,\bar{c}}_{\emptyset,\emptyset}\cdot \varpi(\wt{e}_{i,k}),\ \mathrm{where}\
    \wt{e}_{i,k}=
    \begin{cases}
    e_{i,k} & \text{if } i\ne p \\
    q^{-1}e_{i,k}+(q^{-1}-q)\sum_{r=1}^\infty u^r\cdot e_{i,k-r} & \text{if } i=p
    \end{cases}.$$
Let $\varrho$ be the automorphism of $^{'}\ddot{U}^{+,\wedge}$ such that $\varrho(e_{i,k})=\wt{e}_{i,k}$.
Extending $\Psi$ to an isomorphism of completions $\Psi\colon ^{'}\ddot{U}^{+,\wedge}\iso S^\wedge$, we use
$\wt{\varrho}$ to denote the induced automorphism of $S^\wedge$. Clearly $\wt{\varrho}(\Psi(X))=\Psi(\varrho(X))$
and $Y\star \textbf{1}=\textbf{1}\star \wt{\varrho}(Y)$ for any $X\in ^{'}\ddot{U}^{+,\wedge}, Y\in S^\wedge$.
Therefore
\begin{equation*}
  \sum_l F_l\star \textbf{1}\star G_l=0\Leftrightarrow \sum_l \wt{F}_l G_l=0\Leftrightarrow \sum_l \wt{a}_lb_l=0
  \Leftrightarrow \sum_l \varpi(a_l)\cdot L^{p,\bar{c}}_{\emptyset,\emptyset}\cdot \varpi(b_l)=0.
\end{equation*}
Thus, the linear map $\iota\colon S_{1,p}(u)\to \SSS(p,\bar{c})$ is well-defined and injective.
It is clear that $\iota$ is surjective and is an $S$-bimodule homomorphism. This completes the proof.
\end{proof}


\end{document}